\documentclass{amsart}
\usepackage{amssymb}

\usepackage{amscd}
\usepackage{thmdefs}
\usepackage[dvips]{epsfig}

\input{tcilatex}
\theoremstyle{definition}
\theoremstyle{remark}
\numberwithin{equation}{section}

\input tcilatex

\begin{document}
\title[Operators Induced by Graphs]{Operators Induced by Graphs}
\author{Ilwoo Cho and Palle E. T. Jorgensen}
\address{St. Ambrose Univ., Dept. of Math., 518 W. Locust St., Davenport, Iowa,
52803, U. S. A. / Univ. of Iowa, Dept. of Math., 14 McLean Hall, Iowa City,
Iowa, 52242, U. S. A. }
\email{chowoo@sau.edu / jorgen@math.uiowa.edu}
\thanks{The second named author is supported by the U. S. National Science
Foundation.}
\date{May., 2010}
\dedicatory{}
\thanks{}
\maketitle

\begin{abstract}
In this paper, we consider the spectral-theoretic properties of certain
operators induced by given graphs. Self-adjointness, unitary, hyponormality,
and normality of graph-depending operators are considered. As application,
we study the finitely supported operators in the free group factor $L(F_{N})$%
.
\end{abstract}

\strut

\section{Introduction}

Starting with analysis on countable directed graphs $G,$ we introduce
Hilbert spaces $H_{G}$ and a family of weighted operators $T$ on $H_{G}.$
When the \emph{weights} (called \emph{coefficients} later in the present
context) are chosen, $T$ is called a \emph{graph operator}. From its weights
(or coefficients), we define the support $Supp(T)$ of $T.$ In full
generality, it is difficult to identify analytic tools that reflect global
properties of the underlying graph. We will be interested in generic
properties that allow us to study spectral theory of this family of
operators $T.$ The spectral theorem will produce a spectral measure
representation for $T$ provided we can establish normality of $T$;
self-adjointness, unitary, etc. These are the classes of operators that
admit spectral analysis.

In Theorem 3.1, we give a necessary and sufficient condition on $Supp(T)$
for $T$ to be self-adjoint. Our analysis is of interest even in the case
when $G$ is finite. For instance, in Theorems 3.2 and 3.3, with $G$ assumed
finite, we show that there is a vertex-edge correspondence which
charactrerizes to weighted operators $T$ that are unitary. Also, in Section
4, hyponormality and normality of our graph operators are characterized.

\subsection{Overview}

A \emph{graph} is a set of objects called \emph{vertices} (or points or
nodes) connected by links called \emph{edges} (or lines). In a \emph{%
directed graph}, the two directions are counted as being distinct directed
edges (or arcs). A graph is depicted in a diagrammatic form as a set of dots
(for vertices), jointed by curves or line-segments (for edges). Similarly, a
directed graph is depicted in a diagrammatic form as a set of dots jointed
by arrowed curves, where the arrows point the direction of the directed
edges.

Recently, we studied the operator-algebraic structures induced by directed
graphs. A key idea in the study of graph-depending operator algebras is that
every directed graph $G$ induces its corresponding \emph{groupoid} $\Bbb{G},$
called the\emph{\ graph groupoid of} $G$. By considering the algebraic
structure of $\Bbb{G},$ we can determine the groupoid actions $\lambda ,$
acting on Hilbert spaces $H$: We can obtain suitable representations $(H,$ $%
\lambda )$ for $\Bbb{G}.$ And this guarantees the existence of operator
algebras $\mathcal{A}_{G}$ $=$ $\overline{\Bbb{C}[\lambda (\Bbb{G})]}^{w},$
generated by $\Bbb{G}$ (or induced by $G$), contained in the operator
algebras $B(H).$ Indeed, the operator algebras $\mathcal{A}_{G}$ are the
groupoid ($C^{*}$- or $W^{*}$-)subalgebras of $B(H).$

Note that each edge $e$ of $G$ assigns a partial isometry on $H,$ and every
vertex $v$ of $G$ assigns a projection on $H$ (under various different types
of representations of $\Bbb{G}$). We will fix a \emph{canonical
representation\ }$(H_{G},$ $L)$\emph{\ of} $\Bbb{G},$ and construct the
corresponding von Neumann algebra

\begin{center}
$M_{G}$ $=$ $\overline{\Bbb{C}[L(\Bbb{G})]}^{w}$ in $B(H_{G}),$
\end{center}

where $H_{G}$ is the \emph{graph Hilbert space} $l^{2}(\Bbb{G})$. This von
Neumann algebra $M_{G}$ is called \emph{the graph von Neumann algebra of} $%
G. $ (See Section 3.1 below).

In this paper, we are interested in certain elements $T$ of $M_{G}.$ Recall
that, by the definition of graph von Neumann algebras (which are groupoid
von Neumann algebras), if $T$ $\in $ $M_{G},$ then

\begin{center}
$T$ $=$ $\underset{w\in \Bbb{G}}{\sum }$ $t_{w}L_{w}$ with $t_{w}$ $\in $ $%
\Bbb{C}.$
\end{center}

Define the support $Supp(T)$ of $T$ by

\begin{center}
$Supp(T)$ $=$ $\{w$ $\in $ $\Bbb{G}$ $:$ $t_{w}$ $\neq $ $0\}.$
\end{center}

If the support $Supp(T)$ of $T$ is a ``finite'' subset of $\Bbb{G},$ then we
call $T$ a \emph{graph operator}. If $T$ is a graph operator, then the
quantities $t_{w}$, for $w$ $\in $ $Supp(T),$ are called the\emph{\
coefficients of} $T.$

As we see, all graph operators are (finite) linear sums of generating
operators $L_{w}$ of $M_{G},$ for $w$ $\in $ $\Bbb{G}.$ i.e., they are the
operators generated by finite numbers of projections and partial isometries
on $H_{G}.$ We are interested in the operator-theoretical properties of
them; in particular, self-adjointness, the unitary property, hyponormality,
and normality.

In operator theory, such properties are very important in order to
understand the given operators. For instance, if a given operator $T$ is
normal, then $T$ satisfies the conditions in the \emph{spectral mapping
theorem}, and hence the $C^{*}$-algebra generated by $T$ is $*$-isomorphic
to $C(spec(T)),$ the $C^{*}$-algebra consisting of all continuous functions
on the spectrum $spec(T)$ of $T.$

Recall that, for an operator $T,$ the spectrum of $T$, defined by

\begin{center}
$spec(T)$ $\overset{def}{=}$ $\{t$ $\in $ $\Bbb{C}$ $:$ $T$ $-$ $t1_{H}$ is
not invertible$\},$
\end{center}

is a nonempty compact subset of $\Bbb{C}.$

We characterize the self-adjointness, the unitary property, hyponormality,
and normality of graph operators in $M_{G}.$ We show that such
operator-theoretic properties of graph operators are characterized by the
combinatorial property of given graphs and certain analytic data of
coefficients of $T.$ This provides another connection between operator
theory, operator algebra, groupoid theory, and combinatorial graph theory.

\subsection{Motivation and Applications}

As application, we derive the operator-theoretic properties of finitely
supported elements of the free group factors $L(F_{N})$, for $N$ $\in $ $%
\Bbb{N}.$ Recall that the \emph{free group factor} $L(F_{N}),$ for $N$ $\in $
$\Bbb{N},$ is the group von Neumann algebra $\overline{\Bbb{C}[\lambda
(F_{N})]}^{w}$, in $B(l^{2}(F_{N})),$ generated by the free group $F_{N}$
with $N$-generators, where $(l^{2}(F_{N}),$ $\lambda )$ is the \emph{left
regular unitary representation of} $F_{N},$ consisting of the \emph{group
Hilbert space} $l^{2}(F_{N}),$ and the \emph{unitary representation} (which
is a group action) \emph{of} $F_{N}$\emph{\ acting on} $l^{2}(F_{N}).$ It is
possible since the free group factors $L(F_{N})$ are $*$-isomorphic to the
graph von Neumann algebras $M_{O_{N}}$ of the one-vertex-$N$-loop-edge
graphs $O_{N},$ for all $N$ $\in $ $\Bbb{N}$ $\cup $ $\{\infty \}$ (See
Section 5 below and [5]).\strut

Recall that a von Neumann algebra $\mathcal{M}$ in $B(H)$ is a \emph{factor}%
, if its $W^{*}$-subalgebra $\mathcal{M}^{\prime }$ $\cap $ $\mathcal{M}$ is 
$*$-isomorphic to $\Bbb{C}$ (or $\Bbb{C}$ $\cdot $ $1_{\mathcal{M}}$), where

\begin{center}
$\mathcal{M}^{\prime }$ $\overset{def}{=}$ $\{x$ $\in $ $B(H)$ $:$ $xm$ $=$ $%
mx,$ $\forall $ $m$ $\in $ $\mathcal{M}\}.$
\end{center}

It is well-known that a group $\Gamma $ is an i. c. c (or an infinite
conjugacy class) group, if and only if the corresponding group von Neumann
algebra $L(\Gamma )$ is a factor. Since every free group $F_{N}$ is i. c.
c., the group von Neumann algebra $L(F_{N})$ is a factor. So, we call $%
L(F_{N}),$ the free group factors.

The study of free group factors, itself, is very interesting and important
in operator algebra. We are interested in the operator-theoretic properties
of each element of the fixed free group factor.

We can check that the free group factors $L(F_{N})$ and the graph von
Neumann algebras $M_{O_{N}}$ of the one-vertex-$N$-loop-edge graphs $O_{N}$
are $*$-isomorphic. This provides a motivation for our application in
Section 5. More precisely, the analysis of finitely supported operators in $%
L(F_{N})$ is the study of graph operators in $M_{O_{N}},$ since there are
one-to-one correspondence between finitely supported operators in $L(F_{N}),$
and graph operators in $M_{O_{N}}.$

\section{Definitions and Background}

Starting with a graph $G,$ to understand the operator theory, we must
introduce a Hilbert space $H_{G}$ naturally coming from $G.$ Our approach is
as follows: From $G,$ introduce an enveloping groupoid $\Bbb{G}$ and an
associated involutive algebra $\mathcal{A}_{G}.$ We then introduce a
conditional expectation $E$ of $\mathcal{A}_{G}$ onto the subalgebra $%
\mathcal{D}_{G}$ of diagonal elements. To get a representation of $\mathcal{A%
}_{G}$ and an associated Hilbert space $H_{G},$ we then use the Stinespring
construction on $E$ (e.g., see [14]). \strut In this section, we introduce
the concepts and definitions we will use.\strut \strut 

\subsection{Graph Groupoids}

Let $G$ be a directed graph with its vertex set $V(G)$ and its edge set $%
E(G).$ Let $e$ $\in $ $E(G)$ be an edge connecting a vertex $v_{1}$ to a
vertex $v_{2}.$ Then we write $e$ $=$ $v_{1}$ $e$ $v_{2},$ for emphasizing
the initial vertex $v_{1}$ of $e$ and the terminal vertex $v_{2}$ of $e.$

For a fixed graph $G,$ we can define the oppositely directed graph $G^{-1},$
with $V(G^{-1})$ $=$ $V(G)$ and $E(G^{-1})$ $=$ $\{e^{-1}$ $:$ $e$ $\in $ $%
E(G)\},$ where each element $e^{-1}$ of $E(G^{-1})$ satisfies that

\begin{center}
$e$ $=$ $v_{1}$ $e$ $v_{2}$ in $E(G)$, with $v_{1},$ $v_{2}$ $\in $ $V(G),$
\end{center}

if and only if

\begin{center}
$e^{-1}$ $=$ $v_{2}$ $e^{-1}$ $v_{1},$ in $E(G^{-1}).$
\end{center}

This opposite directed edge $e^{-1}$ $\in $ $E(G^{-1})$ of $e$ $\in $ $E(G)$
is called the \emph{shadow of} $e.$ Also, this new graph $G^{-1}$, induced
by $G,$ is said to be the \emph{shadow of} $G.$ It is clear that $%
(G^{-1})^{-1}$ $=$ $G.$\strut

Define the \emph{shadowed graph} $\widehat{G}$ of $G$ by a directed graph
with its vertex set

\begin{center}
$V(\widehat{G})$ $=$ $V(G)$ $=$ $V(G^{-1})$
\end{center}

and its edge set

\begin{center}
$E(\widehat{G})$ $=$ $E(G)$ $\cup $ $E(G^{-1})$,
\end{center}

where $G^{-1}$ is the \emph{shadow} of $G$.

We say that two edges $e_{1}$ $=$ $v_{1}$ $e_{1}$ $v_{1}^{\prime }$ and $%
e_{2}$ $=$ $v_{2}$ $e_{2}$ $v_{2}^{\prime }$ are \emph{admissible}, if $%
v_{1}^{\prime }$ $=$ $v_{2},$ equivalently, the finite path $e_{1}$ $e_{2}$
is well-defined on $\widehat{G}.$ Similarly, if $w_{1}$ and $w_{2}$ are
finite paths on $G,$ then we say $w_{1}$ and $w_{2}$ are \emph{admissible},
if $w_{1}$ $w_{2}$ is a well-defined finite path on $G,$ too. Similar to the
edge case, if a finite path $w$ has its initial vertex $v$ and its terminal
vertex $v^{\prime },$ then we write $w$ $=$ $v_{1}$ $w$ $v_{2}.$ Notice that
every admissible finite path is a word in $E(\widehat{G}).$ Denote the set
of all finite path by $FP(\widehat{G}).$ Then $FP(\widehat{G})$ is the
subset of the set $E(\widehat{G})^{*},$ consisting of all finite words in $E(%
\widehat{G}).$

Suppose we take a part

\begin{center}
$
\begin{array}{lll}
& \bullet & \overset{e_{3}}{\longrightarrow }\cdot \cdot \cdot \\ 
& \uparrow & _{e_{2}} \\ 
\cdot \cdot \cdot \underset{e_{1}}{\longrightarrow } & \bullet & 
\end{array}
$
\end{center}

in a graph $G$ or in the shadowed graph $\widehat{G},$ where $e_{1},$ $%
e_{2}, $ $e_{3}$ are edges of $G,$ respectively of $\widehat{G}$. Then the
above admissibility shows that the edges $e_{1}$ and $e_{2}$ are admissible,
since we can obtain a finite path $e_{1}e_{2},$ however, the edges $e_{1}$
and $e_{3}$ are not admissible, since a finite path $e_{1}$ $e_{3}$ is
undefined.

We can construct the \emph{free semigroupoid} $\Bbb{F}^{+}(\widehat{G})$ of
the shadowed graph $\widehat{G},$ as the union of all vertices in $V(%
\widehat{G})$ $=$ $V(G)$ $=$ $V(G^{-1})$ and admissible words in $FP(%
\widehat{G}),$ equipped with its binary operation, the \emph{admissibility}$%
. $ Naturally, we assume that $\Bbb{F}^{+}\Bbb{(}\widehat{G})$ contains the 
\emph{empty word} $\emptyset ,$ as the representative of all undefined (or
non-admissible) finite words in $E(\widehat{G})$.

Remark that some free semigroupoid $\Bbb{F}^{+}\Bbb{(}\widehat{G})$ of $%
\widehat{G}$ does not contain the empty word; for instance, if a graph $G$
is a one-vertex-multi-edge graph, then the shadowed graph $\widehat{G}$ of $%
G $ is also a one-vertex-multi-edge graph too, and hence its free
semigroupoid $\Bbb{F}^{+}(\widehat{G})$ does not have the empty word.
However, in general, if $\left| V(G)\right| $ $>$ $1,$ then $\Bbb{F}^{+}(%
\widehat{G})$ always contain the empty word. Thus, if there is no confusion,
we always assume the empty word $\emptyset $ is contained in the free
semigroupoid $\Bbb{F}^{+}(\widehat{G})$ of $\widehat{G}.$

\begin{definition}
By defining the \emph{reduction }(RR) on $\Bbb{F}^{+}(\widehat{G}),$ we
define the graph groupoid $\Bbb{G}$ of a given graph $G,$ by the subset of $%
\Bbb{F}^{+}(\widehat{G}),$ consisting of all ``reduced'' finite paths on $%
\widehat{G},$ with the inherited admissibility on $\Bbb{F}^{+}(\widehat{G})$
under (RR), where the \emph{reduction} (RR) on $\Bbb{G}$ is\ \strut as
follows:

(RR)\qquad $\qquad \qquad \qquad w$ $w^{-1}$ $=$ $v$ and $w^{-1}w$ $=$ $%
v^{\prime },$\strut 

for all $w$ $=$ $v$ $w$ $v^{\prime }$ $\in $ $\Bbb{G},$ with $v,$ $v^{\prime
}$ $\in $ $V(\widehat{G}).$
\end{definition}

Such a graph groupoid $\Bbb{G}$ is indeed a categorial groupoid with its
base $V(\widehat{G})$ (See \textbf{Appendix A}).

\subsection{Canonical Representation of Graph Groupoids}

\strut Let $G$ be a given countable connected directed graph with its graph
groupoid $\Bbb{G}.$ Then we can define the (pure algebraic) algebra $%
\mathcal{A}_{G}$ of $\Bbb{G}$ by a vector space over $\Bbb{C},$ consisting
of all linear combinations of elements of $\Bbb{G},$ i.e.,

\begin{center}
$\mathcal{A}_{G}$ $\overset{def}{=}$ $\Bbb{C}$ $\cup $ $\left( \underset{k=1%
}{\overset{\infty }{\cup }}\left\{ \sum_{j=1}^{k}t_{j}w_{j}\left| 
\begin{array}{c}
w_{j}\in \Bbb{G},\text{ }t_{j}\in \Bbb{C}, \\ 
j=1,...,k
\end{array}
\right. \right\} \right) ,$
\end{center}

under the usual addition ($+$), and the multiplication ($\cdot $), dictated
by the admissibility on $\Bbb{G}.$ Define now a unary operation ($*$) on $%
\mathcal{A}_{G}$ by

\begin{center}
$\sum_{j=1}^{k}$ $t_{j}$ $w_{j}$ $\in $ $\mathcal{A}_{G}$ $\longmapsto $ $%
\sum_{j=1}^{k}$ $\overline{t_{j}}$ $w_{j}^{-1}$ $\in $ $\mathcal{A}_{G},$
\end{center}

where $\overline{z}$ means the conjugate of $z,$ for all $z$ $\in $ $\Bbb{C}%
, $ and of course $w^{-1}$ means the shadow of $w,$ for all $w$ $\in $ $\Bbb{%
G}.$ We call this unary operation ($*$), the \emph{adjoint} (or the \emph{%
shadow}) on $\mathcal{A}_{G}.$ Then the vector space $\mathcal{A}_{G},$
equipped with the adjoint ($*$), is a well-defined (algebraic) $*$-algebra.

Now, define a $*$-subalgebra $\mathcal{D}_{G}$ of $\mathcal{A}_{G}$ by

\begin{center}
$\mathcal{D}_{G}$ $\overset{def}{=}$ $\Bbb{C}$ $\cup $ $\left( \underset{k=1%
}{\overset{\infty }{\cup }}\left\{ \sum_{j=1}^{n}t_{j}\text{ }v_{j}\left| 
\begin{array}{c}
v_{j}\in V(\widehat{G}),\text{ }t_{j}\in \Bbb{C}, \\ 
j=1,...,k
\end{array}
\right. \right\} \right) .$
\end{center}

This $*$-algebra $\mathcal{D}_{G}$ acts like the diagonal of $\mathcal{A}%
_{G},$ so we call $\mathcal{D}_{G},$ the \emph{diagonal} ($*$-)\emph{%
subalgebra of} $\mathcal{A}_{G}.$

\subsubsection{The Hilbert Space $H_{G}$}

Below, we identify the canonical Hilbert space $H_{G}.$ The algebra $%
\mathcal{A}_{G}$ is represented by bounded linear operators acting on $H_{G}.
$ The representation is induced by the canonical conditional expectation,
via the Stinespring construction (e.g., see [14]).

We can construct a (algebraic $*$-)conditional expectation

\begin{center}
$E$ $:$ $\mathcal{A}_{G}$ $\rightarrow $ $\mathcal{D}_{G}$
\end{center}

by

(2.2.1)

\begin{center}
$E\left( \underset{w\in X}{\sum }\text{ }t_{w}w\right) $ $\overset{def}{=}$ $%
\underset{v\in X\cap V(\widehat{G})}{\sum }$ $t_{v}$ $v,$
\end{center}

for all $\underset{w\in X}{\sum }$ $t_{w}w$ $\in $ $\mathcal{A}_{G},$ where $%
X$ means a finite subset of $\Bbb{G}.$

Since the conditional expectation $F$ is completely positive under a
suitable topology on $\mathcal{A}_{G}$, we may apply the Stinespring's
construction. i.e., the diagonal subalgebra $\mathcal{D}_{G}$ is represented
as the $l^{2}$-space, $l^{2}(V(\widehat{G})),$ by the concatenation. Then we
can obtain the Hilbert space $H_{G},$

\begin{center}
$H_{G}$ $\overset{def}{=}$ the Stinespring space of $\mathcal{A}_{G}$ over $%
\mathcal{D}_{G},$ by $F,$
\end{center}

containing $l^{2}(V(\widehat{G})).$ i.e., if $\pi _{(E,\mathcal{D}_{G})}$ is
the Stinespring representation of $\mathcal{A}_{G},$ acting on $l^{2}(V(%
\widehat{G})),$ then

\begin{center}
$H_{G}$ $=$ $\pi _{(E,\mathcal{D}_{G})}\left( \mathcal{A}_{G}\right) .\strut 
$
\end{center}

This Stinespring space $H_{G}$ is the Hilbert space with its \emph{inner
product} $<,>_{G}$ satisfying that:

\begin{center}
$<h,$ $\pi _{(E,\mathcal{D}_{G})}(a)$ $k$ $>_{G}$ $=$ $<h,$ $E(a)$ $k$ $%
>_{2},$
\end{center}

for all $h,$ $k$ $\in $ $l^{2}(V(\widehat{G})),$ for all $a$ $\in $ $%
\mathcal{A}_{G},$ where $<,>_{2}$ is the inner product on $l^{2}(V(\widehat{G%
})).$\strut

i.e., The Stinespring space $H_{G}$ is the norm closure of $\mathcal{A}_{G},$
by the norm,

(2.2.2)

\begin{center}
$\left\| \sum_{j=1}^{n}\text{ }w_{i}\otimes h_{i}\right\| _{G}^{2}$ $=$ $%
\sum_{i=1}^{n}$ $\sum_{k=1}^{n}$ $<h_{i},$ $E(w_{i}^{*}w_{k})$ $h_{k}$ $%
>_{2},$
\end{center}

induced by the Stinespring inner product $<,>_{G}$ on $\mathcal{A}_{G},$ for
all $a_{i}$ $\in $ $\mathcal{A}_{G},$ $h_{i}$ $\in $ $l^{2}(V(\widehat{G})),$
for all $n$ $\in $ $\Bbb{N}.$

\begin{definition}
We call this Stinespring space $H_{G},$ the graph Hilbert space of $\Bbb{G}$
(or of $G$).
\end{definition}

\strut Denote the Hilbert space element $\pi _{(E,\mathcal{D}_{G})}(w)$ by $%
\xi _{w}$ in the graph Hilbert space $H_{G},$ for all $w$ $\in $ $\Bbb{G},$
with the identification,

\begin{center}
$\xi _{\emptyset }$ $=$ $0_{H_{G}},$ the zero vector in $H_{G},$
\end{center}

where $\emptyset $ is the empty word (if exists) of $\Bbb{G}.$ We can check
that the subset $\{\xi _{w}$ $:$ $w$ $\in $ $\Bbb{G}\}$ of $H_{G}$ satisfies
the following \emph{multiplication rule}:

\begin{center}
$\xi _{w_{1}}$ $\xi _{w_{2}}$ $=$ $\xi _{w_{1}w_{2}},$ on $H_{G},$
\end{center}

for all $w_{1},$ $w_{2}$ $\in $ $\Bbb{G}.$ Thus, we can define the \emph{%
canonical multiplication operators }$L_{w}$\emph{\ on} $H_{G}$, satisfying
that

\begin{center}
$L_{w}$ $\xi _{w^{\prime }}$ $\overset{def}{=}$ $\xi _{w}$ $\xi _{w^{\prime
}}$ $=$ $\xi _{ww^{\prime }},$
\end{center}

for all $w,$ $w^{\prime }$ $\in $ $\Bbb{G}.$ The existence of such
multiplication operators $L_{w}$'s guarantees the existence of a groupoid
action $L$ of $\Bbb{G},$ acting on $H_{G}$;

\begin{center}
$L$ $:$ $w$ $\in $ $\Bbb{G}$ $\longmapsto $ $L(w)$ $\overset{def}{=}$ $L_{w}$
$\in $ $B(H_{G}).$
\end{center}

This action $L$ of $\Bbb{G}$ is called the \emph{canonical groupoid action
of }$\Bbb{G}$\emph{\ on} $H_{G}.$

\subsubsection{\strut The Operators $L_{w}$}

Let $w$ and $w_{i}$ denote reduced finite paths in $FP_{r}(\widehat{G}),$
for $i$ $\in $ $\Bbb{N},$ equivalently, they are the reduced words in the
edge set $E(\widehat{G}),$ under the reduction (RR). Consider

(2.2.3)

\begin{center}
$L_{w}\left( \underset{i}{\sum }w_{i}\otimes h_{i}\right) $ $=$ $\underset{i%
}{\sum }$ $ww_{i}$ $\otimes $ $h_{i},$
\end{center}

for $h_{i}$ $\in $ $l^{2}(\Bbb{N})$. Here, the element $\underset{i}{\sum }$ 
$w_{i}$ $\otimes $ $h_{i}$ denotes a finite sum of tensors in $\mathcal{A}%
_{G}$. And $ww_{i}$ in (2.2.3) means concatenation of finite words. With the
conditional expectation $E$ $:$ $\mathcal{A}_{G}$ $\rightarrow $ $\mathcal{D}%
_{G}$ (See (2.2.1) above), we get the \emph{Stinespring representation} $%
(H_{G},$ $\pi _{(E,\mathcal{D}_{G})}),$ and the operators

$\pi _{(E,\mathcal{D}_{G})}(w)$ $:$ $\mathcal{H}_{G}$ $\rightarrow $ $%
\mathcal{H}_{G}$

obtained from (2.2.3) by passing to the quotient and completion as in
Definition 2.2. To simplify terminology, in the sequel, we will simply write 
$L_{w}$ for the operator $\pi _{(E,\mathcal{D}_{G})}(w).$

\subsubsection{\strut Graph von Neumann Algebras}

Let $G$, $\Bbb{G},$ and $H_{G}$ be given as above. And let $\{L_{w}$ $:$ $w$ 
$\in $ $\Bbb{G}\}$ the multiplication operators on $H_{G}$, where $L$ is the
canonical groupoid action of $\Bbb{G}.$

\begin{definition}
Let $G$ be a countable directed graph with its graph groupoid $\Bbb{G}.$ The
pair $(H_{G},$ $L)$ of the graph Hilbert space $H_{G}$ and the canonical
groupoid action $L$ of $\Bbb{G}$ is called the canonical representation of $%
\Bbb{G}$. The corresponding groupoid von Neumann algebra

\begin{center}
$M_{G}$ $\overset{def}{=}$ $\overline{\Bbb{C}[L(\Bbb{G})]}^{w},$
\end{center}

generated by $\Bbb{G}$ (equivalently, by $L(\Bbb{G})$ $=$ $\{L_{w}$ $:$ $w$ $%
\in $ $\Bbb{G}\}$), as a $W^{*}$-subalgebra of $B(H_{G})$, is called the
graph von Neumann algebra of $G.$
\end{definition}

\strut \strut We can check that the generating operators $L_{w}$'s of the
graph von Neumann algebra $M_{G}$ of $G$ satisfies that:

\begin{center}
$L_{w}^{*}$ $=$ $L_{w^{-1}},$ for all $w$ $\in $ $\Bbb{G},$
\end{center}

and

\begin{center}
$L_{w_{1}}L_{w_{2}}$ $=$ $L_{w_{1}w_{2}},$ for all $w_{1},$ $w_{2}$ $\in $ $%
\Bbb{G}.$\\[0pt]
\end{center}

It is easy to check that if $v$ is a vertex in $\Bbb{G},$ then the graph
operator $L_{v}$ is a projection, since

\begin{center}
$L_{v}^{*}$ $=$ $L_{v^{-1}}$ $=$ $L_{v}$ $=$ $L_{v^{2}}$ $=$ $L_{v}^{2}.$
\end{center}

Thus, by the reduction (RR) on $\Bbb{G},$ we can conclude that if $w$ is a
nonempty reduced finite path in $FP_{r}(\widehat{G}),$ then the operator $%
L_{w}$ is a partial isometry, since

\begin{center}
$L_{w}^{*}$ $L_{w}$ $=$ $L_{w^{-1}w}$,
\end{center}

and $w^{-1}w$ is a vertex, and hence $L_{w}^{*}L_{w}$ is a projection on $%
H_{G}.$\strut

\section{\strut Self-Adjointness and Unitary Property}

In this section, we introduce our main objects of this paper: canonical
representations of graph groupoids, graph von Neumann algebras, and graph
operators. And we study the self-adjointness of graph operators, and the
unitary property of them. We can realize that the self-adjointness and the
unitary property of graph operators are characterized by the combinatorial
property (admissibility) of given graphs, and certain analytic data of
coefficients of the operators.

Section 3.1 introduces the graph operators, and the theorem in Section 3.2
yields the structure of the graph operators that are self-adjoint; and
Section 3.3, the unitary case. The different geometries of $G$ and the
associated operators reflect different spectral representations. Section 4
below covers of normal and hyponormal graph operators. Finally, Section 5
takes up the case when the algebra is one of the free group factors (e.g.,
see [15], and [16]).

\subsection{Graph Operators}

Let $G$ be a graph with its graph groupoid $\Bbb{G},$ and let $M_{G}$ $=$ $%
\overline{\Bbb{C}[L(\Bbb{G})]}^{w}$ be the graph von Neumann algebra of $G$
in $B(H_{G}),$ where $(H_{G},$ $L)$ is the canonical representation of $\Bbb{%
G}.$ Since $M_{G}$ is a groupoid von Neumann algebra generated by $\Bbb{G},$
every element $T$ of $M_{G}$ satisfies the expansion,

\begin{center}
$T$ $=$ $\underset{w\in \Bbb{G}}{\sum }$ $t_{w}$ $L_{w},$ with $t_{w}$ $\in $
$\Bbb{C}.$
\end{center}

For the given operator $T$ $\in $ $M_{G},$ having the above expansion,
define the subset $Supp(T)$ of $\Bbb{G}$ by

\begin{center}
$Supp(T)$ $\overset{def}{=}$ $\{w$ $\in $ $\Bbb{G}$ $:$ $t_{w}$ $\neq $ $%
0\}. $
\end{center}

This subset $Supp(T)$ of $\Bbb{G}$ is called the \emph{support of} $T.$

\begin{definition}
Let $T$ be an element of the graph von Neumann algebra $M_{G}$ of a given
graph $G,$ and let $Supp(T)$ be the support of $T.$ If $Supp(T)$ is a finite
set, then we call $T$ a graph operator (on $H_{G}$).
\end{definition}

i.e., graph operators are the finitely supported operators on $H_{G}.$

In the rest of this section, we will consider a very specific example, but
very interesting, where a given graph $G$ is an infinite linear graph,

\begin{center}
$G$ $=$ \quad $\bullet \longrightarrow \bullet \longrightarrow \bullet
\longrightarrow \cdot \cdot \cdot .$
\end{center}

\strut \strut

We want to investigate the matrix forms of (which is unitarily equivalent
to) graph operators. Instead of determining the matrix forms of graph
operators, acting on the graph Hilbert space $H_{G}$, we consider the matrix
forms of them, acting on the subspace $l^{2}(V(\widehat{G})),$ embedded in
the graph Hilbert space $H_{G}.$

For convenience, we let

\begin{center}
$V(G)$ $=$ $\Bbb{N},$ and $E(G)$ $=$ $\{(j,$ $j$ $+$ $1)$ $:$ $j$ $\in $ $%
\Bbb{N}\},$
\end{center}

i.e.,

\begin{center}
$G$ $=$ \quad $\underset{1}{\bullet }\overset{(1,2)}{\longrightarrow }%
\underset{2}{\bullet }\overset{(2,3)}{\longrightarrow }\underset{3}{\bullet }%
\overset{(3,4)}{\longrightarrow }$ $\cdot \cdot \cdot .$
\end{center}

\strut

Then, we can check that

\begin{center}
$l^{2}\left( V(\widehat{G})\right) $ $\overset{\text{Hilbert}}{=}$ $l^{2}(%
\Bbb{N})$ in $H_{G}.$
\end{center}

So, we can assign the graph operator $L_{j}$ to the infinite matrix

\begin{center}
$
\begin{array}{ll}
\qquad \qquad \qquad j\text{-th} &  \\ 
\left( 
\begin{array}{lllllll}
0 &  &  &  &  &  & 0 \\ 
& \ddots &  &  &  &  &  \\ 
&  & 0 &  &  &  &  \\ 
&  &  & 1 &  &  &  \\ 
&  &  &  & 0 &  &  \\ 
&  &  &  &  & \ddots &  \\ 
0 &  &  &  &  &  & 
\end{array}
\right) & j\text{-th,}
\end{array}
$
\end{center}

on $l^{2}(\Bbb{N}),$ for all $j$ $\in $ $\Bbb{N}$ $=$ $V(\widehat{G}),$ and
we assign the graph operator $L_{(j,\text{ }j+1)}$ to the infinite matrix

\begin{center}
\strut $
\begin{array}{ll}
\qquad \qquad \qquad j\text{-th} &  \\ 
\left( 
\begin{array}{lllllll}
0 &  &  &  &  &  & 0 \\ 
0 & \ddots &  &  &  &  &  \\ 
& \ddots & 0 &  &  &  &  \\ 
&  & 0 & 1 & 1 &  &  \\ 
&  &  & 0 & 0 &  &  \\ 
&  &  &  & 0 & \ddots &  \\ 
0 &  &  &  &  & \ddots & 
\end{array}
\right) & j\text{-th,}
\end{array}
$
\end{center}

on $l^{2}(\Bbb{N}),$ for all $j$ $\in $ $\Bbb{N}.$ More precisely, we can
assign

\begin{center}
$L_{j}$ $\in $ $M_{G}$ $\longleftrightarrow $ $\mid j><j\mid $ $\in $ $%
B\left( l^{2}(\Bbb{N})\right) $
\end{center}

and

\begin{center}
$L_{(j,j+1)}$ $\in $ $M_{G}\longleftrightarrow $ $\mid j><j\mid +\mid
j><j+1\mid $ $\in $ $B\left( l^{2}(\Bbb{N})\right) ,$
\end{center}

where $\mid j>$ means the Dirac operators, for all $j$ $\in $ $\Bbb{N}.$

We use Dirac's notation for rank-one operators, i.e.,

\begin{center}
$\mid u><v\mid x$ $=$ $<v,$ $x>$ $u,$
\end{center}

defined for vectors $u,$ $v,$ $x$ in a fixed Hilbert space having its inner
product $<,>.$

So, for a reduced finite path $w$ $=$ $e_{i_{1}}$ $e_{i_{2}}$ ... $e_{i_{k}}$
$\in $ $\Bbb{G},$ with $e_{i_{j}}$ $=$ $(i_{j},$ $i_{j}$ $+$ $1)$ $\in $ $E(%
\widehat{G}),$ where

\begin{center}
$i_{j+1}$ $=$ $i_{j}$ $+$ $1,$ for $j$ $=$ $1,$ ..., $k$ $-$ $1,$
\end{center}

the graph operator $L_{w}$ is determined as a matrix,

\begin{center}
$A_{e_{i_{1}}}$ $+$ $A_{e_{i_{2}}}$ $+$ ... $+$ $A_{e_{i_{k}}},$
\end{center}

where $A_{e_{i_{j}}}$ are the infinite matrices (on $l^{2}(\Bbb{N})$) of the
graph operators $L_{e_{i_{j}}}.$

\strut For instance, the self-adjoint operator

\begin{center}
$
\begin{array}{ll}
L_{(j,j+1)}+L_{(j,j+1)}^{*} & =L_{(j,j+1)}+L_{(j+1,j)} \\ 
& =2\mid j><j\mid +\mid j><j+1\mid \\ 
& \qquad \qquad +\mid j+1><j\mid
\end{array}
$
\end{center}

has its matrix form

\begin{center}
$\left( 
\begin{array}{lllllll}
0 & 0 &  &  &  &  &  \\ 
0 & \ddots & \ddots &  &  &  &  \\ 
& \ddots & 0 & 0 &  &  &  \\ 
&  & 0 & 2 & 1 &  &  \\ 
&  &  & 1 & 0 & 0 &  \\ 
&  &  &  & 0 & \ddots & \ddots \\ 
&  &  &  &  & \ddots & 
\end{array}
\right) ,$
\end{center}

\strut on $l^{2}(\Bbb{N})$ $=$ $l^{2}\left( V(\widehat{G})\right) .$ So,
more generally, the self-adjoint operator $L_{w}$ $+$ $L_{w}^{*}$, for $w$ $%
\in $ $\Bbb{G},$ becomes a certain self-adjoint Toeplitz operator on $l^{2}(%
\Bbb{N}),$ because $l^{2}(\Bbb{N})$ is Hilbert-space isomorphic to the Hardy
space $H^{2}(\Bbb{T}),$ equipped with the Haar measure, where $\Bbb{T}$ is
the unit circle in $\Bbb{C}$ (e.g., see [3], [12], and [13]).

\subsection{\strut Self-Adjoint Graph Operators}

The operator-theoretic properties of (bounded linear) operators; the
self-adjointness, the unitary properties, the hyponormality, and the
normality are briefly introduced in \textbf{Appendix B}.

In this section, we will consider the self-adjointness of graph operators.
Let $G$ be a graph with its graph groupoid $\Bbb{G},$ and let $M_{G}$ be the
graph von Neumann algebra of $G.$ Take a graph operator $T$ in $M_{G},$

\begin{center}
$T$ $=$ $\underset{w\in Supp(T)}{\sum }$ $t_{w}$ $L_{w},$ with $t_{w}$ $\in $
$\Bbb{C}.$
\end{center}

The following theorem characterize the self-adjointness of $T$.

\begin{theorem}
Let $T$ $\in $ $M_{G}$ be a given graph operator. Then $T$ is self-adjoint,
if and only if there exists ``a'' subset $X$ of $Supp(T)$ such that

\begin{center}
$Supp(T)$ $\cap $ $FP_{r}(\widehat{G})$ $=$ $X$ $\sqcup $ $X^{-1},$
\end{center}

where $\sqcup $ means the disjoint union, and

\begin{center}
$t_{x}$ $=$ $\overline{t_{x^{-1}}},$ for all $x$ $\in $ $X,$
\end{center}

where $X^{-1}$ $\overset{def}{=}$ $\{x^{-1}$ $:$ $x$ $\in $ $X\},$ and $%
\overline{z}$ means the conjugate of $z,$ for all $z$ $\in $ $\Bbb{C},$ and

\begin{center}
$t_{v}$ $\in $ $\Bbb{R},$ for all $v$ $\in $ $Supp(T)$ $\cap $ $V(\widehat{G}%
).$
\end{center}
\end{theorem}

\begin{proof}
($\Leftarrow $) Assume that $T$ $=$ $\underset{w\in Supp(T)}{\sum }$ $t_{w}$ 
$L_{w}$ is a graph operator in $M_{G},$ and suppose there exists a subset $X$
of

\begin{center}
$Supp_{V}^{c}(\widehat{G})$ $\overset{denote}{=}$ $Supp(T)$ $\cap $ $FP_{r}(%
\widehat{G})$
\end{center}

such that

\begin{center}
$Supp_{V}^{c}(T)$ $=$ $X$ $\sqcup $ $X^{-1},$
\end{center}

and

\begin{center}
$t_{x}$ $=$ $\overline{t_{x^{-1}}},$ for all $x$ $\in $ $X.$
\end{center}

Also, assume that $t_{v}$ $\in $ $\Bbb{R},$ for all elements $v$ in

\begin{center}
$Supp_{V}(T)$ $\overset{denote}{=}$ $Supp(T)$ $\cap $ $V(\widehat{G}).$
\end{center}

Then the operator $T$ can be re-written by

\begin{center}
$T$ $=$ $\underset{v\in Supp_{V}(T)}{\sum }$ $t_{v}L_{v}$ $+$ $\underset{%
x\in X}{\sum }$ $t_{x}$ $L_{x}$ $+$ $\underset{x^{-1}\in X^{-1}}{\sum }$ $%
t_{x^{-1}}$ $L_{x^{-1}}.$
\end{center}

Moreover, we can have that

$\qquad T^{*}$ $=$ $\left( \underset{v\in Supp_{V}(T)}{\sum }t_{v}L_{v}+%
\underset{x\in X}{\sum }\text{ }t_{x}L_{x}+\underset{x^{-1}\in X^{-1}}{\sum }%
t_{x^{-1}}L_{x^{-1}}\right) ^{*}$

$\qquad \qquad =$ $\underset{v\in Supp_{V}(T)}{\sum }$ $\overline{t_{v}}$ $%
L_{v^{-1}}$ $+$ $\underset{x\in X}{\sum }$ $\overline{t_{x}}$ $L_{x^{-1}}$ $%
+ $ $\underset{x^{-1}\in X^{-1}}{\sum }$ $\overline{t_{x^{-1}}}$ $L_{x}$

$\qquad \qquad =$ $\underset{v\in Supp_{V}(T)}{\sum }$ $t_{v}L_{v}$ $+$ $+$ $%
\underset{x\in X}{\sum }$ $\overline{t_{x}}$ $L_{x^{-1}}$ $+$ $\underset{%
x^{-1}\in X^{-1}}{\sum }$ $\overline{t_{x^{-1}}}$ $L_{x}$

since $t_{v}$ $\in $ $\Bbb{R},$ and $L_{v}$ are projections for all $v$ $\in 
$ $V(\widehat{G})$

$\qquad \qquad =$ $\underset{v\in Supp_{V}(T)}{\sum }$ $t_{v}L_{v}$ $+$ $%
\underset{x\in X}{\sum }$ $t_{x^{-1}}$ $L_{x^{-1}}$ $+$ $\underset{x^{-1}\in
X^{-1}}{\sum }$ $t_{x}L_{x}$

since $t_{x}$ $=$ $\overline{t_{x^{-1}}},$ for all $x$ $\in $ $X$

$\qquad \qquad =$ $\underset{v\in Supp_{V}(T)}{\sum }$ $t_{v}L_{v}$ $+$ $%
\underset{x^{-1}\in X^{-1}}{\sum }$ $t_{x^{-1}}L_{x^{-1}}$ $+$ $\underset{%
x\in X}{\sum }$ $t_{x}$ $L_{x}$

since $Supp(T)$ $=$ $X$ $\sqcup $ $X^{-1}$

$\qquad \qquad =$ $T.$

Therefore, under hypothesis, the adjoint $T^{*}$ of $T$ is identical to $T,$
itself, and hence the element $T$ of $M_{G}$ is self-adjoint.

($\Rightarrow $) Let $T$ $\in $ $M_{G}$ be a self-adjoint graph operator,
i.e., $T$ satisfies $T^{*}$ $=$ $T.$ Then

\begin{center}
$
\begin{array}{ll}
T^{*} & =\left( \underset{w\in Supp(T)}{\sum }\text{ }t_{w}L_{w}\right) ^{*}
\\ 
& =\underset{w\in Supp(T)}{\sum }\overline{t_{w}}L_{w^{-1}}\overset{(\star )%
}{=}\underset{w\in Supp(T)}{\sum }t_{w}L_{w} \\ 
& =T.
\end{array}
$
\end{center}

To satisfy the above equality ($\star $), we must have

\begin{center}
$Supp(T^{*})$ $=$ $Supp(T).$
\end{center}

Notice that the support $Supp(T^{*})$ of the adjoint $T^{*}$ of $T$ satisfies

\begin{center}
$Supp(T^{*})$ $=$ $Supp(T)^{-1},$ in $\Bbb{G}$
\end{center}

So, the self-adjointness of $T$ guarantees

\begin{center}
$Supp(T)$ $=$ $Supp(T)^{-1}$ in $\Bbb{G}.$
\end{center}

Therefore, since $Supp(T)$ is self-adjoint, in the sense that $Supp(T)$ is
identical to $Supp(T)^{-1},$ there must exists a subset $X$ of $%
Supp_{V}^{c}(T)$ such that

\begin{center}
$Supp_{V}^{c}(T)$ $=$ $X$ $\sqcup $ $X^{-1},$
\end{center}

because the following set equality always holds true;

\begin{center}
$Supp_{V}(T)^{-1}$ $=$ $Supp_{V}(T)$
\end{center}

(since $V(\widehat{G})^{-1}$ $=$ $V(\widehat{G})$ $=$ $V(G)$ $=$ $V(G^{-1})$%
).

Now, let $X$ be a subset satisfying the above set equality,

\begin{center}
$Supp_{V}^{c}(T)$ $=$ $X$ $\sqcup $ $X^{-1},$ in $\Bbb{G}.$
\end{center}

For a fixed element $x$ $\in $ $X,$ the coefficient $t_{x}$ of $T$ has its
corresponding coefficient $t_{x^{-1}}$ of $T.$ Assume now that there exists
at least one element $x_{0}$ $\in $ $X,$ such that

\begin{center}
$t_{x_{0}}$ $\neq $ $\overline{t_{x_{0}^{-1}}}$ in $\Bbb{C}.$
\end{center}

Then the summand $t_{x_{0}}L_{x_{0}}$ of $T$ satisfies that

\begin{center}
$(t_{x_{0}}L_{x_{0}})^{*}$ $=$ $\overline{t_{x_{0}}}$ $L_{x_{0}^{-1}}$ $\neq 
$ $t_{x_{0}^{-1}}$ $L_{x_{0}^{-1}},$
\end{center}

and hence $T^{*}$ $\neq $ $T$ on $H_{G}.$ This contradicts our
self-adjointness of $T.$

Therefore, if $T$ is self-adjoint, then there exists a unique subset $X$ of
the support $Supp(T)$ of $T$ such that

\begin{center}
$Supp_{V}^{c}(T)$ $=$ $X$ $\sqcup $ $X^{-1},$
\end{center}

and

\begin{center}
$t_{x}$ $=$ $\overline{t_{x^{-1}}},$ for all $x$ $\in $ $X.$
\end{center}

Similarly, assume that there exists at least one $v_{0}$ $\in $ $%
Supp_{V}(T), $ such that $t_{v_{0}}$ $\in $ $\Bbb{C}$ $\setminus $ $\Bbb{R}.$
Then the summand $t_{v_{0}}$ $L_{v_{0}}$ of $T$ satisfies that

\begin{center}
$(t_{v_{0}}L_{v_{0}})^{*}$ $=$ $\overline{t_{v_{0}}}$ $L_{v_{0}^{-1}}$ $=$ $%
\overline{t_{v_{0}}}$ $L_{v_{0}}$ $\neq $ $t_{v_{0}}$ $L_{v_{0}},$
\end{center}

since $\overline{t_{v_{0}}}$ $\neq $ $t_{v_{0}},$ whenever $t_{v_{0}}$ $%
\notin $ $\Bbb{R}$ in $\Bbb{C}.$ This also contradicts our assumption that $%
T $ is self-adjoint.
\end{proof}

\strut The above theorem characterizes the self-adjointness of graph
operators $T$ by the classification of the support $Supp(T),$ and the
coefficients of $T.$ This is interesting since the self-adjointness of graph
operators are determined by the combinatorial data represented by the
elements of the supports (or the admissibility of graph groupoids of given
graphs), and the simple analytic data of coefficients.

\begin{example}
Let $G$ be a graph,

\strut

\begin{center}
$G$ $=$ \quad $_{v_{1}}\bullet \overset{e_{1}}{\underset{e_{2}}{%
\rightrightarrows }}\underset{v_{2}}{\bullet }\overset{e_{3}}{\leftarrow }%
\bullet _{v_{3}}.$
\end{center}

\strut Let

\begin{center}
$T_{1}$ $=$ $t_{v_{1}}$ $L_{v_{1}}$ $+$ $t_{e_{1}}L_{e_{1}}$ $+$ $%
t_{e_{1}^{-1}}L_{e_{1}^{-1}}$ $+$ $%
t_{e_{3}e_{2}^{-1}}L_{e_{3}e_{2}^{-1}}+t_{e_{2}e_{3}^{-1}}L_{e_{2}e_{3}^{-1}}, 
$
\end{center}

and

\begin{center}
$T_{2}$ $=$ $t_{e_{2}}L_{e_{2}}$ $+$ $t_{e_{3}}L_{e_{3}}$ $+$ $%
t_{e_{3}^{-1}}L_{e_{3}^{-1}},$
\end{center}

in $M_{G}.$ Then we can check the self-adjointness of $T_{1}$ and $T_{2}$
immediately by the above theorem. First, consider the self-adjointness of $%
T_{1}.$ We can see that

\begin{center}
$Supp_{V}(T_{1})$ $=$ $\{v_{1}\},$ and $Supp_{V}^{c}(T_{1})$ $=$ $\{e_{1},$ $%
e_{1}^{-1},$ $e_{3}e_{2}^{-1},$ $e_{2}e_{3}^{-1}\},$
\end{center}

in $Supp(T_{1}).$ So, there exists a subset $X$ of $Supp(T_{1}),$

\begin{center}
$X$ $=$ $\{e_{1},$ $e_{3}e_{2}^{-1}\},$ having $X^{-1}$ $=$ $\{e_{1}^{-1},$ $%
e_{2}e_{3}^{-1}\},$
\end{center}

satisfying that

\begin{center}
$Supp_{V}^{c}(T_{1})$ $=$ $X$ $\sqcup $ $X^{-1}.$
\end{center}

(From this example, we can realize that the existence of $X$ is not uniquely
determined. For instance, we may take a set $Y,$

\begin{center}
$Y$ $=$ $\{e_{1}^{-1},$ $e_{3}e_{2}^{-1}\},$ having $Y^{-1}$ $=$ $\{e_{1},$ $%
e_{2}e_{3}^{-1}\},$
\end{center}

satisfying $Supp_{V}^{c}(T_{1})$ $=$ $Y$ $\sqcup $ $Y^{-1}.$)

So, the graph operator $T_{1}$ is self-adjoint on $H_{G},$ if and only if

\begin{center}
$t_{v_{1}}$ $\in $ $\Bbb{R},$
\end{center}

and

\begin{center}
$t_{e_{1}}$ $=$ $\overline{t_{e_{1}^{-1}}},$ and $t_{e_{3}e_{2}^{-1}}$ $=$ $%
\overline{t_{e_{2}e_{3}^{-1}}},$ in $\Bbb{C}.$
\end{center}

Also, for an operator $T_{2},$ we can immediately check that $T_{2}$ never
be self-adjoint on $H_{G},$ because

\begin{center}
$Supp(T_{2})$ $=$ $Supp_{V}^{c}(T_{2})$ $=$ $\{e_{2},$ $e_{3},$ $%
e_{3}^{-1}\},$
\end{center}

and there does not exist a subset $X,$ satisfying

\begin{center}
$Supp_{V}^{c}(T_{2})$ $=$ $X$ $\sqcup $ $X^{-1}.$
\end{center}

Therefore, a graph operator $T_{2}$ is not self-adjoint on $H_{G}.$\strut
\end{example}

\subsection{\strut Unitary Graph Operators}

\strut In this section, we will consider the unitary graph operators in the
given graph von Neumann algebra $M_{G}$ of a connected directed graph $G.$
To consider the unitary property of graph operators, we will restrict our
interests to the case where a given connected graph $G$ is a finite graph.
Recall that a graph $G$ is \emph{finite}, if

\begin{center}
$\left| V(G)\right| $ $<$ $\infty ,$ and $\left| E(G)\right| $ $<$ $\infty .$
\end{center}

\strut

\textbf{Assumption} In this section, we assume all given graphs are
``finite.'' $\square $

\strut

The reason we only consider finite graphs to study the unitary property of
graph operators is that: we want to determine the identity operator $i_{d}$
on the graph Hilbert space $H_{G},$ easily.

Notice that the identity operator $i_{d}$ in $B(H_{G})$ is identified with
the element

\begin{center}
$1_{M_{G}}$ $=$ $\underset{v\in V(\widehat{G})}{\sum }$ $L_{v}$ in $M_{G}.$
\end{center}

\begin{remark}
\strut Remark that, even though the given graph $K$ is ``infinite,'' in
particular, $\left| V(K)\right| $ $=$ $\infty ,$ the identity element $%
1_{M_{K}}$ of the corresponding graph von Neumann algebra $M_{K}$ is the
operator $\underset{v\in V(\widehat{K})}{\sum }$ $L_{v},$ under topology.
So, the identity element $1_{M_{K}}$ is not finitely supported. Therefore,
we can verify that a finitely supported element $T$ of $M_{K}$ (which is our
graph operator) would not be unitary, since the Cartesian product

\begin{center}
$Supp(T)^{r_{1}}\times ...\times Supp(T)^{r_{n}},$
\end{center}

where

\begin{center}
$(r_{1},$ ..., $r_{n})$ $\in $ $\{\pm 1\}^{n}$,
\end{center}

is a finite set, for all $n$ $\in $ $\Bbb{N}.$ Thus, we restrict our
interests to the case where we have ``finite'' graphs.
\end{remark}

Let $1_{M_{G}}$ be the identity element of the graph von Neumann algebra $%
M_{G}$ of a finite graph $G.$ Then, an operator $U$ on $H_{G}$ is unitary,
if and only if

\begin{center}
$U^{*}U$ $=$ $1_{M_{G}}$ $=$ $UU^{*},$
\end{center}

by definition, and hence, equivalently, $U^{*}$ $=$ $U^{-1},$ where $U^{-1}$
means the \emph{inverse of} $U.$

Now, let's fix a graph operator

\begin{center}
$T$ $=$ $\underset{w\in Supp(T)}{\sum }$ $t_{w}$ $L_{w}$ in $M_{G}.$
\end{center}

\strut Then the adjoint $T^{*}$ of $T$ is

\begin{center}
$T^{*}$ $=$ $\underset{w\in Supp(T)}{\sum }$ $\overline{t_{w}}$ $L_{w^{-1}}$
in $M_{G}.$
\end{center}

Thus the products $T^{*}T$ of $TT^{*}$ are

\begin{center}
$T^{*}T$ $=$ $\underset{(w_{1},w_{2})\in Supp(T)^{2}}{\sum }$ $\overline{%
t_{w_{1}}}$ $t_{w_{2}}$ $L_{w_{1}^{-1}w_{2}}$
\end{center}

and

\begin{center}
$TT^{*}$ $=$ $\underset{(y_{1},y_{2})\in Supp(T)^{2}}{\sum }$ $t_{y_{1}}%
\overline{t_{y_{2}}}$ $L_{y_{1}y_{2}^{-1}},$
\end{center}

respectively, where

\begin{center}
$Supp(T)^{2}$ $\overset{def}{=}$ $Supp(T)$ $\times $ $Supp(T).$
\end{center}

\begin{definition}
Let $X$ be a subset of the graph groupoid $\Bbb{G}$ of $G.$ We say that this
subset $X$ is \emph{alternatively disconnected}, if it satisfies that:

(i)\ \ \ $\left| X\right| $ $\geq $ $2,$

(ii)\ \ for any pair $(w_{1},$ $w_{2})$ of ``distinct'' elements $w_{1}$ and 
$w_{2}$ of

\begin{center}
$X$ $\cap $ $FP_{r}(\widehat{G})$ (if it exists, or if it is nonempty),
\end{center}

neither ``$w_{1}^{-1}$ and $w_{2},$'' nor ``$w_{1}$ and $w_{2}^{-1}$'' is
admissible in $\Bbb{G}.$
\end{definition}

\strut Let $G$ be a finite graph,

\begin{center}
$G$ $=$ \quad $_{v_{1}}\bullet \overset{e_{1}}{\longleftarrow }\underset{%
v_{2}}{\bullet }\overset{e_{2}}{\longrightarrow }\bullet _{v_{3}}$
\end{center}

\strut

and let

\begin{center}
$X_{1}$ $=$ $\{v_{1},$ $e_{1},$ $e_{2}\},$ $X_{2}$ $=$ $\{e_{1}^{-1},$ $%
e_{2},$ $v_{3}\},$ $X_{3}$ $=$ $\{v_{2},$ $v_{3}\}$
\end{center}

be given subsets of the graph groupoid $\Bbb{G}$ of $G.$ Then, we can check
that the subset $X_{1}$ is not alternatively disconnected, because it does
not satisfy the condition (ii) of the definition. i.e., both ``$e_{1}^{-1}$
and $e_{2},$'' and ``$e_{2}^{-1}$ and $e_{1}"$ are admissible in $\Bbb{G}$.
Also, we can see the subset $X_{2}$ is alternatively disconnected. Indeed,
neither ``$e_{1}$ $=$ $(e_{1}^{-1})^{-1}$ and $e_{2},$'' nor ``$e_{2}^{-1}$
and $e_{1}^{-1}$'' is admissible in $\Bbb{G}.$ Clearly, the subset $X_{3}$
is alternatively disconnected, since it satisfies the conditions (i) and
(ii) of the above theorem.

\strut Also, all vertex sets of (finite) graphs are alternatively
disconnected in the above sense.

Now, let's go back to our main interest of this section. To become a graph
operator $T$ of $M_{G}$ to be unitary, both operators $T^{*}T$ and $TT^{*}$
must be the identity element

\begin{center}
$1_{M_{G}}$ $=$ $\underset{v\in V(\widehat{G})}{\sum }$ $L_{v}$ in $M_{G}.$
\end{center}

Thus we can obtain the following characterization.

\begin{theorem}
Let $G$ be a finite graph with

\begin{center}
$\left| V(G)\right| $ $\geq $ $2,$
\end{center}

and let $T$ $\in $ $M_{G}$ be a graph operator with its support $Supp(T).$
Then $T$ is unitary, if and only if

(i)$\ \ \ Supp(T)$ is alternatively disconnected,

(ii)\ \ the support $Supp(T)$ satisfies

\begin{center}
$\left( Supp(T)\right) ^{-1}\left( Supp(T)\right) $ $=$ $V(\widehat{G}),$
\end{center}

where $X^{-1}X$ $\overset{def}{=}$ $\{w_{1}^{-1}w_{2}$ $:$ $w_{1},$ $w_{2}$ $%
\in $ $X\},$ for all $X$ $\subset $ $\Bbb{G},$ and

(iii) the coefficients of $T$ satisfy

\begin{center}
$\underset{w\in Supp(T),\text{ }w^{-1}w=v}{\sum }$ $\left| t_{w}\right| ^{2}$
$=$ $1,$ for all $v$ $\in $ $V(\widehat{G}),$
\end{center}

in $\Bbb{C}.$
\end{theorem}

\begin{proof}
Assume that the given graph operator $T$ is unitary on $H_{G}.$ Then, by
definition,

(3.3.1)

\begin{center}
$T^{*}T$ $=$ $\underset{(w_{1},w_{2})\in Supp(T)^{2}}{\sum }$ $\overline{%
t_{w_{1}}}$ $t_{w_{2}}$ $L_{w_{1}^{-1}w_{2}}$ $=$ $\underset{v\in V(\widehat{%
G})}{\sum }$ $L_{v}$ $=$ $1_{M_{G}},$
\end{center}

and

(3.3.2)

\begin{center}
$TT^{*}$ $=$ $\underset{(y_{1},y_{2})\in Supp(T)^{2}}{\sum }$ $t_{y_{1}}%
\overline{t_{y_{2}}}$ $L_{y_{1}y_{2}^{-1}}$ $=$ $\underset{v\in V(\widehat{G}%
)}{\sum }$ $L_{v}$ $=$ $1_{M_{G}},$
\end{center}

in the graph von Neumann algebra $M_{G}$ of a finite connected graph $G.$
Notice here that, if there exists a pair $(w_{1},$ $w_{2})$ of distinct
elements $w_{1}$ $\neq $ $w_{2}$ in $Supp(T),$ such that $w_{1}^{-1}w_{2}$ $%
\neq $ $\emptyset ,$ equivalently, $w_{1}^{-1}$ and $w_{2}$ are admissible
in $\Bbb{G},$ then there exists an nonzero summand

\begin{center}
$\overline{t_{w_{1}}}$ $t_{w_{2}}$ $L_{w_{1}^{-1}w_{2}}$
\end{center}

in (3.3.1). By the distinctness of $w_{1}$ and $w_{2},$ and by the
assumption $w_{1}^{-1}w_{2}$ $\neq $ $\emptyset ,$ the element $%
w_{1}^{-1}w_{2}$ must be a nonempty reduced finite path in $\Bbb{G}.$ This
shows that the first equality (3.3.1) does not hold, and hence it
contradicts our unitary property of $T.$

Similarly, if $w_{1}w_{2}^{-1}$ $\neq $ $\emptyset ,$ then there exists an
nonzero summand

\begin{center}
$t_{w_{1}}$ $\overline{t_{w_{2}}}$ $L_{w_{1}w_{2}^{-1}}$
\end{center}

in (3.3.2), and hence this term breaks the unitary property of $T,$ which
contradicts our assumption for $T.$

Therefore, to satisfy the unitary property of $T,$ the support $Supp(T)$ of $%
T$ is alternatively disconnected, i.e., for any pair $(w_{1},$ $w_{2})$ of
distinct elements in $Supp(T),$ neither ``$w_{1}^{-1}$ and $w_{2},$'' nor ``$%
w_{1}$ and $w_{2}^{-1}$'' is admissible in $\Bbb{G}.$ Under the alternative
disconnectedness of $Supp(T),$ we can obtain the alternating form of the
left-hand side of (3.3.1):

(3.3.3)

\begin{center}
$
\begin{array}{ll}
T^{*}T & =\underset{(w_{1},w_{2})\in Supp(T)^{2},\text{ }w_{1}=w_{2}}{\sum }%
\overline{t_{w_{1}}}t_{w_{2}}L_{w_{1}^{-1}w_{2}} \\ 
& =\underset{w\in Supp(T)}{\sum }\overline{t_{w}}t_{w}L_{w^{-1}w} \\ 
& =\underset{w\in Supp(T)}{\sum }\left| t_{w}\right| ^{2}L_{w^{-1}w}.
\end{array}
$
\end{center}

Remark here that $w^{-1}w$ $\in $ $V(\widehat{G}),$ for all $w$ $\in $ $\Bbb{%
G}.$ By (3.3.3), we can re-write that $T$ is unitary if and only if

(3.3.4)

\begin{center}
$T^{*}T$ $=$ $\underset{w\in Supp(T)}{\sum }$ $\left| t_{w}\right| ^{2}$ $%
L_{w^{-1}w}$ $=$ $\underset{v\in V(\widehat{G})}{\sum }$ $L_{v}$ $=$ $%
1_{M_{G}},$
\end{center}

by the finiteness of $G.$ And the second equality of (3.3.4) can be refined
as follows:

(3.3.5)

\begin{center}
$
\begin{array}{ll}
\underset{w\in Supp(T)}{\sum }\left| t_{w}\right| ^{2}L_{w^{-1}w} & =%
\underset{v\in V(\widehat{G})}{\sum }\left( \underset{w\in
Supp(T),\,w^{-1}w=v}{\sum }\text{ }\left| t_{w}\right| ^{2}\text{ }%
L_{v}\right) \\ 
& =\underset{v\in V(\widehat{G})}{\sum }\left( \underset{w\in
Supp(T),\,w^{-1}w=v}{\sum }\left| t_{w}\right| ^{2}\right) L_{v}=1_{M_{G}}.
\end{array}
$
\end{center}

Therefore, by (3.3.5), the support of $T$ must satisfy

(3.3.6)

\begin{center}
$\left( Supp(T)^{-1}\right) \left( Supp(T)\right) $ $=$ $V(\widehat{G}),$
\end{center}

and, under the alternative disconnectedness (3.3.6) of $T$, the coefficients
of $T$ must satisfy

(3.3.7)

\begin{center}
$\underset{w\in Supp(T),\,w^{-1}w=v}{\sum }$ $\left| t_{w}\right| ^{2}$ $=$ $%
1,$ for all $v$ $\in $ $V(\widehat{G}),$
\end{center}

in $\Bbb{C}$, where

\begin{center}
$X^{-1}X$ $\overset{def}{=}$ $\{w_{1}^{-1}w_{2}$ $:$ $w_{1},$ $w_{2}$ $\in $ 
$X\},$ for all $X$ $\subset $ $\Bbb{G}.$
\end{center}

i.e., we can obtain that $T^{*}T$ $=$ $1_{M_{G}},$ if and only if the
support $Supp(T)$ is alternatively disconnected, and it satisfies (3.3.6),
and the coefficients of $T$ satisfy (3.3.7).

Similar to the above observation, we can get that $TT^{*}$ $=$ $1_{M_{G}},$
if and only if $Supp(T)$ is alternatively disconnected, and it satisfies
(3.3.6), and the coefficients of $T$ satisfies

(3.3.8)

\begin{center}
$\underset{y\in Supp(T),\,yy^{-1}=x}{\sum }$ $\left| t_{y}\right| ^{2}$ $=$ $%
1,$ for all $x$ $\in $ $V(\widehat{G}).$
\end{center}

However, it is easy to check that the conditions (3.3.7) and (3.3.8) are
equivalent, because there exists a bijection $g,$

\begin{center}
$g$ $:$ $w$ $\in $ $Supp(T)$ $\longmapsto $ $w^{-1}$ $\in $ $Supp(T)^{-1}.$
\end{center}

Therefore, we can conclude that the graph operator $T$ is unitary, if and
only if the support $Supp(T)$ of $T$ is alternatively disconnected, and it
also satisfies the conditions (3.3.6), and the coefficients of $T$ satisfy
(3.3.7) (or (3.3.8)).\strut
\end{proof}

Similar to the self-adjointness of graph operators, the unitary property of
graph operators are also determined by the admissibility on the graph
groupoids of given graphs and certain conditions on coefficients of the
operators.

\begin{remark}
\strut In the proof of the above theorem (the unitary characterization of
graph operators), where

\begin{center}
$\left| V(G)\right| $ $\geq $ $2,$
\end{center}

the alternative disconnectedness is crucial. Since $\left| V(G)\right| $ $>$ 
$1,$ all generating operators $L_{w}$'s of the graph von Neumann algebra $%
M_{G}$ are partial isometries. Moreover, the products $L_{w_{1}}$ ... $%
L_{w_{n}},$ for all $n$ $\in $ $\Bbb{N},$ are partial isometries, whose
initial and final spaces are ``not'' identified with the graph Hilbert space 
$H_{G}.$ Thus, to satisfy the unitary property, the products $L_{w_{1}w_{2}}$
either in $T^{*}T$ or in $TT^{*}$ must be the zero operator, whenever $w_{1}$
$\neq $ $w_{2}$ in $\Bbb{G}.$
\end{remark}

In the rest of this paper, we will consider following two examples.

\begin{example}
Let $T$ $=$ $\underset{v\in V(\widehat{G})}{\sum }$ $t_{v}$ $L_{v}$ be a
graph operator in $M_{G}.$ Then it is unitary, if and only if (i) $t_{v}$ $%
\neq $ $0,$ and (ii) $\left| t_{v}\right| ^{2}$ $=$ $1,$ for all $v$ $\in $ $%
V(\widehat{G}).$
\end{example}

\begin{example}
Let $G$ be a connected finite graph,

\strut

\begin{center}
$G$ $=$ \qquad $_{v_{1}}\bullet \overset{e_{1}}{\longrightarrow }\underset{%
v_{2}}{\bullet }\overset{e_{2}}{\longrightarrow }\underset{v_{3}}{\bullet }%
\overset{e_{3}}{\longrightarrow }\bullet _{v_{4}}$.
\end{center}

\strut

Let $T_{1}$ $=$ $t_{v_{1}}L_{v_{1}}$ $+$ $%
t_{e_{2}^{-1}}L_{e_{2}^{-1}}+t_{e_{2}e_{3}}$ $L_{e_{2}e_{3}}$ be a given
graph operator in the graph von Neumann algebra $M_{G}$ of $G.$ We can check
that

\begin{center}
$Supp(T_{1})$ $=$ $\{v_{1},$ $e_{2}^{-1},$ $e_{2}e_{3}\},$
\end{center}

and hence

\begin{center}
$\Pi _{1}$ $=$ $\left( Supp(T_{1})^{-1}\right) \left( Supp(T_{1})\right) $ $%
= $ $\{v_{1},$ $v_{2},$ $v_{4}\}.$
\end{center}

So, $\Pi _{1}$ $\neq $ $V(G)$ $=$ $V(\widehat{G}).$ Therefore, this graph
operator $T_{1}$ is not unitary.

Now, let $T_{2}$ $=$ $t_{v_{1}}L_{v_{1}}$ $+$ $t_{v_{3}}L_{v_{3}}$ $+$ $%
t_{e_{2}^{-1}}L_{e_{2}^{-1}}$ $+$ $t_{e_{2}e_{3}}L_{e_{2}e_{3}}.$ Then the
support

\begin{center}
$Supp(T_{2})$ $=$ $\{v_{1},$ $v_{3},$ $e_{2}^{-1},$ $e_{2}e_{3}\}$ of $T_{2}$
\end{center}

satisfies that

\begin{center}
$\Pi _{2}$ $=$ $\left( Supp(T_{2})^{-1}\right) \left( Supp(T_{2})\right) $ $%
= $ $\{v_{1},$ $v_{2},$ $v_{3},$ $v_{4}\}$ $=$ $V(\widehat{G}).$
\end{center}

Moreover, all the pairs $(w_{1},$ $w_{2})$ of distinct elements $w_{1}$ and $%
w_{2}$ of $Supp(T_{2})$ are alternatively disconnected. For instance,

\begin{center}
$\left( e_{2}^{-1}\right) ^{-1}(e_{2}e_{3})$ $=$ $e_{2}^{2}e_{3}$ $=$ $%
\emptyset $ $e_{3}$ $=$ $\emptyset ,$

$v_{3}^{-1}$ $e_{2}$ $=$ $v_{3}$ $e_{2}$ $=$ $\emptyset ,$ and $e_{2}^{-1}$ $%
v_{3}$ $=$ $\emptyset ,$
\end{center}

etc. Therefore, we can obtain that the operator $T_{2}$ is unitary on the
graph Hilbert space $H_{G},$ if and only if

\begin{center}
$\underset{w\in Supp(T_{2}),\,w^{-1}w=v_{1}}{\sum }$ $\left| t_{w}\right|
^{2}$ $=$ $\left| t_{v_{1}}\right| ^{2}$ $=$ $1,$

$\underset{w\in Supp(T_{2}),\,w^{-1}w=v_{2}}{\sum }$ $\left| t_{w}\right|
^{2}$ $=$ $\left| t_{v_{2}}\right| ^{2}$ $+$ $\left| t_{e_{2}^{-1}}\right|
^{2}$ $=$ $1,$

$\underset{w\in Supp(T_{2}),\,w^{-1}w=v_{3}}{\sum }$ $\left| t_{w}\right|
^{2}$ $=$ $\left| t_{v_{3}}\right| ^{2}$ $=$ $1,$
\end{center}

and

\begin{center}
$\underset{w\in Supp(T_{2}),\,w^{-1}w=v_{4}}{\sum }$ $\left| t_{w}\right|
^{2}$ $=$ $\left| t_{e_{2}e_{3}}\right| ^{2}$ $=$ $1.$
\end{center}

Simply, $T_{2}$ is unitary, if and only if

\begin{center}
$\left| t_{v_{1}}\right| ^{2}$ $=$ $\left| t_{v_{3}}\right| ^{2}$ $=$ $%
\left| t_{e_{2}e_{3}}\right| ^{2}$ $=$ $1,$ and $\left| t_{v_{2}}\right|
^{2} $ $+$ $\left| t_{e_{2}^{-1}}\right| ^{2}$ $=$ $1.$
\end{center}

\strut
\end{example}

\strut The above unitary characterization of graph operators (induced by
finite graphs) is in fact incomplete, since we did not consider the case
where a given graph $G$ satisfies $\left| V(G)\right| $ $=$ $1.$ If a finite
graph $G$ has only one vertex $v_{0},$ then it is graph-isomorphic to the
one-vertex-$\left| E(G)\right| $-multi-loop-edge graph $O_{\left|
E(G)\right| }.$ To make our unitary characterization of graph operators
complete, we need the following theorem.

\begin{theorem}
Let $O_{n}$ be the one-vertex-$n$-loop-edge graph with its graph groupoid $%
\Bbb{O}_{n},$ having its unique vertex $v_{O},$ and let $M_{O_{n}}$ be the
graph von Neumann algebra of $O_{n},$ for $n$ $\in $ $\Bbb{N}.$ Let

\begin{center}
$T$ $=$ $\underset{w\in Supp(T)}{\sum }$ $t_{w}$ $L_{w}$ $\in $ $M_{O_{n}}$
\end{center}

be a fixed graph operator. Then $T$ is unitary, if and only if

(i)$\ \ \ \left( Supp(T)\right) ^{-1}\left( Supp(T)\right) $ $=$ $\{v_{O}\},$

(ii)\ \ the coefficients $\{t_{w}$ $:$ $w$ $\in $ $Supp(T)\}$ of $T$
satisfies

\begin{center}
$\underset{(w_{1},w_{2})\in Supp(T)^{2}}{\sum }$ $\overline{t_{w_{1}}}$ $%
t_{w_{2}}$ $=$ $1.$
\end{center}
\end{theorem}

\begin{proof}
($\Leftarrow $) \strut Assume that a fixed graph operator $T$ of $M_{O_{n}}$
satisfies both conditions (i) and (ii). Then we can obtain that

$\qquad \qquad T^{*}T$ $=$ $\underset{(w_{1}^{-1},w_{2})\in \left(
Supp(T)\right) ^{-1}\times Supp(T)}{\sum }$ $\left( \overline{t_{w_{1}}}%
\text{ }t_{w_{2}}\right) $ $L_{w_{1}^{-1}w_{2}}$

$\qquad \qquad \qquad =$ $\underset{(w_{1},w_{2})\in Supp(T)^{2}}{\sum }$ $%
\left( \overline{t_{w_{1}}}\text{ }t_{w_{2}}\right) $ $L_{v_{O}}$

by (i)

$\qquad \qquad \qquad =$ $\left( \underset{(w_{1},w_{2})}{\sum }\left( 
\overline{t_{w_{1}}}\text{ }t_{w_{2}}\right) \right) $ $L_{v_{O}}$ $=$ $%
L_{v_{O}}$

by (ii). Notice that, by definition, $L_{v_{O}}$ is the identity element of $%
M_{O_{n}}.$ i.e., $L_{v_{O}}$ $=$ $1_{M_{O_{n}}}.$ Therefore, we have that

(3.3.9)

\begin{center}
$T^{*}T$ $=$ $1_{M_{O_{n}}}$.
\end{center}

Consider now $TT^{*}.$ Observe that, under hypothesis,

$\qquad \qquad TT^{*}$ $=$ $\underset{(w_{1},w_{2}^{-1})\in Supp(T)\times
\left( Supp(T)\right) ^{-1}}{\sum }$ $\left( t_{w_{1}}\overline{t_{w_{2}}}%
\right) $ $L_{w_{1}w_{2}^{-1}}$

$\qquad \qquad \qquad =$ $\underset{(w_{1},w_{2})\in Supp(T)}{\sum }\left(
t_{w_{1}}\overline{t_{w_{2}}}\right) $ $L_{v_{O}}$

by (i), and by the fact that:

\begin{center}
$
\begin{array}{ll}
\left( Supp(T)\right) ^{-1}\left( Supp(T)\right) & =\{v_{O}\}=\{v_{O}^{-1}\}
\\ 
& =\left( Supp(T)\right) \left( Supp(T)\right) ^{-1},
\end{array}
$
\end{center}

thus we can have

$\qquad \qquad \qquad =$ $\left( \underset{(w_{1},w_{2})}{\sum }(t_{w_{1}}%
\overline{t_{w_{2}}})\right) $ $L_{v_{O}}$ $=$ $\overline{\left( \underset{%
(w_{1},w_{2})}{\sum }(\overline{t_{w_{1}}}\text{ }t_{w_{2}})\right) }$ $%
L_{v_{O}}$

$\qquad \qquad \qquad =$ $\overline{1}$ $\cdot $ $L_{v_{O}}$ $=$ $1$ $\cdot $
$L_{v_{O}}$ $=$ $L_{v_{O}},$

by (ii). Therefore, we obtain that

(3.3.10)

\begin{center}
$TT^{*}$ $=$ $1_{M_{O_{n}}}.$
\end{center}

So, by (3.3.9) and (3.3.10), this graph operator $T$ is unitary.

($\Rightarrow $) Suppose a graph operator $T$ of $M_{O_{n}}$ is unitary.
Assume that $T$ does not satisfy the condition (i). Then we can pick a pair

\begin{center}
$(w_{1},$ $w_{2})$ $\in $ $(Supp(T))^{2},$
\end{center}

such that $w_{1}$ $\neq $ $w_{2},$ and $w_{1}^{-1}w_{2}$ $\neq $ $v_{O},$
equivalently, $w_{1}^{-1}w_{2}$ $\in $ $FP_{r}(\widehat{O_{n}}).$ This means
that the product $T^{*}T$ of $T^{*}$ and $T$ contains a nonzero summand $%
\overline{t_{w_{1}}}$ $t_{w_{2}}$ $L_{w_{1}^{-1}w_{2}}$. Thus,

\begin{center}
$T^{*}T$ $\neq $ $1_{M_{O_{n}}}$ $=$ $L_{v_{O}}.$
\end{center}

This contradicts our assumption that $T$ is unitary. Assume now that $T$
does not satisfy the condition (ii). Say

\begin{center}
$\underset{(w_{1},w_{2})\in Supp(T)^{2}}{\sum }$ $(\overline{t_{w_{1}}}$ $%
t_{w_{2}})$ $=$ $t_{0}$ $\neq $ $1$, in $\Bbb{C}.$
\end{center}

For convenience, assume $T$ satisfies the condition (i). Then the product $%
T^{*}T$ of $T^{*}$ and $T$ is identical to

\begin{center}
$T^{*}T$ $=$ $t_{0}$ $L_{v_{O}}$ $\neq $ $L_{v_{O}}$ $=$ $1_{M_{O_{n}}}.$
\end{center}

This contradict the unitary property of $T.$
\end{proof}

\strut The above theorem characterizes the unitary property of graph
operators induced by the one-vertex-multi-loop-edge graphs.

\strut

\textbf{Conclusion (Unitary Characterization of Graph Operators)}

Let $G$ be a finite graph and let $M_{G}$ be the graph von Neumann algebra
of $G.$ Let

\begin{center}
$T$ $=$ $\underset{w\in Supp(T)}{\sum }$ $t_{w}$ $L_{w}$ $\in $ $M_{G}$
\end{center}

by a graph operator.

(3.3.11) Assume that $\left| V(G)\right| $ $=$ $1.$ Then $T$ is unitary, if
and only if

\begin{center}
$\left( Supp(T)\right) ^{-1}\left( Supp(T)\right) $ $=$ $V(G),$
\end{center}

and

\begin{center}
$\underset{(w_{1},w_{2})\in Supp(T)^{2}}{\sum }$ $\overline{t_{w_{1}}}$ $%
t_{w_{2}}$ $=$ $1,$ in $\Bbb{C}.$
\end{center}

\strut

(3.3.12) Assume now that $\left| V(G)\right| $ $=$ $2.$ Then $T$ is unitary,
if and only if $Supp(T)$ is alternatively disconnected, and

\begin{center}
$\left( Supp(T)\right) ^{-1}\left( Supp(T)\right) $ $=$ $V(G),$
\end{center}

and

\begin{center}
$\underset{w\in Supp(T),\,w^{-1}w=v}{\sum }$ $\left| t_{w}\right| ^{2}$ $=$ $%
1,$ for all $v$ $\in $ $V(G).$
\end{center}

$\square $

\section{\strut Normality of Graph Operators}

In this section, we will consider the normality of graph operators. Let $G$
be a connected directed graph with its graph groupoid $\Bbb{G},$ and let $%
M_{G}$ $=$ $\overline{\Bbb{C}[L(\Bbb{G})]}^{w}$ be the graph von Neumann
algebra of $G$ in $B(H_{G}),$ where $(H_{G},$ $L)$ is the canonical
representation of $\Bbb{G},$ consisting of the graph Hilbert space $H_{G}$ $%
= $ $l^{2}(\Bbb{G}),$ and the canonical groupoid action $L$ of $\Bbb{G}.$

We are interested in the normality of graph operators in $M_{G}.$ Recall
that an operator $T$ is \emph{normal}, if $T^{*}T$ $=$ $TT^{*}.$ Before
checking the normality of a graph operator $T$ $\in $ $M_{G},$ we will
consider the hyponormality of $T$ in Section 4.1. Recall that an operator $T$
is \emph{hyponormal}, if $T^{*}T$ $-$ $TT^{*}$ is positive.

The hyponormality characterization of graph operators would give the
normality characterization directly. Notice here that (pure) hyponormal
operators and normal operators have few common analytic properties. So, in
general, we do not know how the hyponormality determines the normality.
However, in our graph-operator case, the hyponormal characterization
determines the normality characterization.

As we have seen in Section 3, the self-adjointness and the unitary property
of graph operators are characterized by the admissibility on the graph
groupoid $\Bbb{G}$ (equivalently, the combinatorial property of $G$ or $%
\widehat{G}$), and certain analytic data of coefficients. We hope to obtain
the similar normality characterization.

\subsection{Hyponormality}

To consider the normality of graph operators, we first characterize the
hyponormality of them. Note that the hyponormality, itself, is interesting
in operator theory (e.g., see [12], and [13]). For instance, the
hyponormality of Toeplitz operators have been studied widely (e.g., See [3],
and cited papers of [3]).

We may understand hyponormality (or co-hyponormality) as the generalized
normality. But keep in mind that hyponormal operators and normal operators
do not share analytic properties much. However, in our case, we can show
that hyponormality of graph operators and normality of graph operators are
combinatorially related.

In this section, we characterize the hyponormality of a given graph operator 
$T,$ in terms of the combinatorial information on a fixed graph groupoid and
the analytic data on the coefficients of $T,$ like in Sections 3.2, and 3.3.

Recall that an operator $T$ is \emph{positive} on a Hilbert space $H,$ if

\begin{center}
\strut $<T\xi ,$ $\xi >\;\geq $ $0,$ for all $\xi $ $\in $ $H.$
\end{center}

Here, $<,>$ means the inner product on $H,$ and $\left\| .\right\| $ means
the corresponding Hilbert norm induced by $<,>.$ If $T$ is a positive
operator on $H,$ we write

\begin{center}
$T$ $\geq $ $0_{H},$
\end{center}

where $0_{H}$ is the zero operator in $B(H).$ When $T_{1}$ and $T_{2}$ are
operators on $H,$ we write

\begin{center}
$T_{1}$ $\geq $ $T_{2},$
\end{center}

if the operator $T_{1}$ $-$ $T_{2}$ is a positive operator on $H.$\strut

Thus, by definition, an operator $T$ is \emph{hyponormal} on $H,$ if and
only if $T^{*}T$ $\geq $ $TT^{*}$ on $H,$ equivalently, the operator

\begin{center}
$S(T)$ $\overset{denote}{=}$ $[T^{*},$ $T]$ $\overset{def}{=}$ $T^{*}T$ $-$ $%
TT^{*}$
\end{center}

is a positive operator on $H,$ where $[A,$ $B]$ means the operator,

\begin{center}
$[A,$ $B]$ $\overset{def}{=}$ $AB$ $-$ $BA$, for all $A,$ $B$ $\in $ $B(H).$
\end{center}

We call the operator $[A,$ $B]$, the \emph{commutator} of $A$ and $B.$ In
particular, if $A$ $=$ $T^{*},$ and $B$ $=$ $T,$ for $T$ $\in $ $B(H),$ the
commutator $[T^{*},$ $T]$ is called the \emph{self-commutator of} $T.$

Notice here that the self-commutator $S(T)$ $=$ $[T^{*},$ $T]$ of every
operator $T$ is self-adjoint on $H.$

Define a two maps $s,$ $r$ $:$ $\Bbb{G}$ $\rightarrow $ $V(\widehat{G})$ by

\begin{center}
$r(w)$ $\overset{def}{=}$ $w^{-1}w,$ and $s(w)$ $\overset{def}{=}$ $ww^{-1},$
\end{center}

for all $w$ $\in $ $\Bbb{G}.$ i.e., these maps $r$ and $s$ are the range map
and the source map of the (graph) groupoid $\Bbb{G},$ in the sense of
Section 2.2.

\strut Now, fix a graph operator

\begin{center}
$T$ $=$ $\underset{w\in Supp(T)}{\sum }$ $t_{w}L_{w}$ $\in $ $M_{G},$
\end{center}

acting on the graph Hilbert space $H_{G}.$ The self-commutator $S(T)$ of $T$
is computed as follows:\strut

$\qquad S(T)$ $=$ $T^{*}T$ $-$ $TT^{*}$

(4.1.1)

$\qquad \qquad =$ $\underset{(w_{1},w_{2})\in Supp(T)^{2}}{\sum }$ $%
\overline{t_{w_{1}}}$ $t_{w_{2}}$ $L_{w_{1}^{-1}w_{2}}$ $-$ $\underset{%
(y_{1},y_{2})\in Supp(T)^{2}}{\sum }$ $t_{y_{1}}$ $\overline{t_{y_{2}}}$ $%
L_{y_{1}y_{2}^{-1}}$

$\qquad \qquad =$ $\underset{(w_{1},w_{2})\in Supp(T)^{2}}{\sum }$ $%
\overline{t_{w_{1}}}$ $t_{w_{2}}\left(
L_{w_{1}^{-1}w_{2}}-L_{w_{2}w_{1}^{-1}}\right) $

(4.1.2)

$\qquad \qquad =$ $\underset{(w_{1},w_{2})\in
Supp(T)^{2},\,w_{1}^{-1}w_{2}\neq \emptyset }{\sum }$ $\overline{t_{w_{1}}}$ 
$t_{w_{2}}$ $\left( L_{w_{1}^{-1}w_{2}}-L_{w_{2}w_{1}^{-1}}\right) $

$\qquad \qquad =$ $\left( \underset{w\in Supp(T)}{\sum }\left| t_{w}\right|
^{2}\left( L_{r(w)}-L_{s(w)}\right) \right) $

$\qquad \qquad \qquad \qquad +$ $\left( \underset{(w_{1},w_{2})\in
Supp(T),\,w_{1}\neq w_{2},\,w_{1}^{-1}w_{2}\neq \emptyset }{\sum }\text{ }%
\overline{t_{w_{1}}}\text{ }t_{w_{2}}\left(
L_{w_{1}^{-1}w_{2}}-L_{w_{2}w_{1}^{-1}}\right) \right) $

where $r(w)$ $=$ $w^{-1}w,$ and $s(w)$ $=$ $ww^{-1},$ for all $w$ $\in $ $%
\Bbb{G}$

$\qquad \qquad =$ $\left( \underset{w\in Supp(T)}{\sum }\left| t_{w}\right|
^{2}\left( L_{r(w)}-L_{s(w)}\right) \right) $

\begin{center}
$+$ $\left( \underset{(w_{1},w_{2})\in Supp(T),\,w_{1}\neq w_{2}^{\pm
1},\,w_{1}^{-1}w_{2}\neq \emptyset }{\sum }\mathcal{S}_{(w_{1},w_{2})}%
\right) ,$
\end{center}

where

\strut

\begin{center}
$\mathcal{S}_{(w_{1},w_{2})}$ $\overset{def}{=}$ $\left( \overline{t_{w_{1}}}%
t_{w_{2}}\left( L_{w_{1}^{-1}w_{2}}-L_{w_{2}w_{1}^{-1}}\right) +t_{w_{1}}%
\overline{t_{w_{2}}}\left( L_{w_{2}^{-1}w_{1}}-L_{w_{1}w_{2}^{-1}}\right)
\right) ,$
\end{center}

\strut \strut

for all $(w_{1},$ $w_{2})$ $\in $ $Supp(T)^{2},$ such that $w_{1}$ $\neq $ $%
w_{2}^{\pm 1}$

\strut

\strut (4.1.3)

$\qquad \qquad =$ $\left( \underset{w\in Supp(T)}{\sum }\left| t_{w}\right|
^{2}\left( L_{r(w)}-L_{s(w)}\right) \right) $

\begin{center}
$
\begin{array}{ll}
+\left( \underset{(w_{1},w_{2})\in Supp(T),\,w_{1}\neq w_{2}^{\pm
1},\,w_{1}^{-1}w_{2}\neq \emptyset }{\sum }\right. & \left( \left( \overline{%
t_{w_{1}}}t_{w_{2}}L_{w_{1}^{-1}w_{2}}+t_{w_{1}}\overline{t_{w_{2}}}%
L_{w_{2}^{-1}w_{1}}\right) \right. \\ 
& \left. \left. \quad -\left( \overline{t_{w_{1}}}%
t_{w_{2}}L_{w_{2}w_{1}^{-1}}+t_{w_{1}}\overline{t_{w_{2}}}%
L_{w_{1}w_{2}^{-1}}\right) \right) \right) .
\end{array}
$
\end{center}

\strut The computation (4.1.3) indeed shows that the self-commutator $S(T)$
is self-adjoint on the graph Hilbert space $H_{G},$ since each summand of
(4.1.3) is self-adjoint. Recall that if two operators are self-adjoint, then
the addition of these two operators is again self-adjoint.

The hyponormality of $T$ is guaranteed by the positivity of the self-adjoint
operator $S(T).$ In general, it is not easy to check when a self-adjoint
operator $S$ is positive, because, for example, it is hard to see when the
spectrum $spec(S)$ (contained in $\Bbb{R}$) is contained in $\Bbb{R}_{0}^{+}$
$=$ $\{r$ $\in $ $\Bbb{R}$ $:$ $r$ $\geq $ $0\}.$

However, in our graph-operator case, we can check the positivity of $S(T)$
of $T,$ by (4.1.1), (4.1.2), (4.1.3), and the computations,

\begin{center}
$<S(T)\xi _{x},$ $\xi _{x}>,$ and $<S(T)$ $\xi _{x},$ $\xi _{y}>,$
\end{center}

for $x,$ $y$ $\in $ $\Bbb{G}$ $\setminus $ $\{\emptyset \}$ (equivalently,
for $\xi _{x},$ $\xi _{y}$ $\in $ $\mathcal{B}_{H_{G}}$ in $H_{G}$), where $%
<,>$ means the inner product on $H_{G}.$ To check the positivity of $S(T),$
we have to show that

\begin{center}
$<S(T)$ $\xi ,$ $\xi >$ $\geq $ $0,$ for all $\xi $ $\in $ $H_{G}.$
\end{center}

Since the collection of vectors

\begin{center}
$\eta $ $=$ $\underset{x\in \Bbb{G}}{\sum }$ $r_{x}$ $\xi _{x}$ $\in $ $%
H_{G},$ with $r_{x}$ $\in $ $\Bbb{C},$
\end{center}

is dense in $H_{G},$ it is enough to show that

\begin{center}
$<S(T)\eta ,$ $\eta >$ $\geq $ $0,$ for all $\eta $ $\in $ $\mathcal{H}_{G},$
\end{center}

where

\begin{center}
$\mathcal{H}_{G}$ $\overset{def}{=}$ $\left\{ \eta =\underset{x\in \Bbb{G}}{%
\sum }r_{x}\xi _{x}\left| r_{x}\in \Bbb{C},\xi _{x}\in \mathcal{B}%
_{H_{G}}\right. \right\} $ $\subseteq $ $H_{G}.$
\end{center}

\begin{lemma}
\strut Let $L_{w}$ $\in $ $M_{G}$ be a generating operator of $M_{G}$
induced by $w$ $\in $ $\Bbb{G}.$ Then

(4.1.4)

\begin{center}
$<L_{w}$ $\xi _{x},$ $\xi _{y}$ $>$ $=$ $\delta _{r(w),\,s(x)}$ $\delta
_{wx,\,y},$
\end{center}

where $\delta $ means the Kronecker delta.
\end{lemma}

\begin{proof}
\strut Compute

$\qquad <L_{w}$ $\xi _{x},$ $\xi _{y}$ $>$ $=$ $<\xi _{wx},$ $\xi _{y}$ $>$

\strut

$\qquad \qquad \qquad =$ $\left\{ 
\begin{array}{ll}
<\xi _{wx},\text{ }\xi _{y}> & \text{if }r(w)=s(x) \\ 
<\xi _{\emptyset },\text{ }\xi _{y}>\text{ }=0 & \text{otherwise}
\end{array}
\right. $

\strut

$\qquad \qquad \qquad =$ $\delta _{r(w),\,s(x)}$ $<\xi _{wx},$ $\xi _{y}$ $>$

\strut

$\qquad \qquad \qquad =$ $\left\{ 
\begin{array}{ll}
\delta _{r(w),\,s(x)}\text{ }\cdot \text{ }1 & \text{if }wx=y \\ 
\delta _{r(w),\,s(x)}\text{ }\cdot \text{ }0 & \text{otherwise}
\end{array}
\right. $

\strut since $\xi _{wx},$ $\xi _{y}$ $\in $ $\mathcal{B}_{H_{G}}$ $\cup $ $%
\{0_{H_{G}}\}$

$\qquad \qquad \qquad =$ $\delta _{r(w),\,s(x)}$ $\delta _{wx,\,y}.$

Therefore,

\begin{center}
$<L_{w}\xi _{x,\,},$ $\xi _{y}>$ $=$ $\delta _{r(w),\,s(x)}$ $\delta
_{wx,\,y},$
\end{center}

for all $w,$ $x,$ $y$ $\in $ $\Bbb{G}.$
\end{proof}

\strut By (4.1.4), we can obtain the following lemma.

\begin{lemma}
Let $T$ $=$ $\underset{w\in Supp(T)}{\sum }$ $t_{w}$ $L_{w}$ $\in $ $M_{G}$
be a graph operator, and let $\xi $ $=$ $\underset{x\in \Bbb{G}}{\sum }$ $%
r_{x}$ $\xi _{x}$ $\in $ $\mathcal{H}_{G}$ be a vector in $H_{G}.$ Then

(4.1.5)

$\qquad <S(T)$ $\xi ,$ $\xi $ $>$

$\qquad \qquad =$ $\underset{(x,y)\in \Bbb{G}^{2}}{\sum }r_{x}\overline{r_{y}%
}$ $\left( \underset{(w_{1},w_{2})\in Supp(T)^{2},\,w_{1}^{-1}w_{2}\neq
\emptyset }{\sum }\overline{t_{w_{1}}}t_{w_{2}}\delta
_{r(w_{1}^{-1}w_{2}),\,s(x)}\delta _{w_{1}^{-1}w_{2}x,\,y}\right. $

$\qquad \qquad \qquad \qquad \qquad \qquad \quad \left. -\underset{%
(y_{1},y_{2})\in Supp(T)^{2},\,y_{1}y_{2}^{-1}\neq \emptyset }{\sum }%
t_{y_{1}}\overline{t_{y_{2}}}\delta _{r(y_{1}y_{2}^{-1}),\,s(x)}\delta
_{y_{1}y_{2}^{-1}x,\,y}\right) .$

$\square $
\end{lemma}

\strut The proof of the above theorem is straightforward, by (4.1.1) and
(4.1.4). Now, we denote the summands

\strut $\qquad r_{x}\overline{r_{y}}$ $\left( \underset{(w_{1},w_{2})\in
Supp(T)^{2},\,w_{1}^{-1}w_{2}\neq \emptyset }{\sum }\overline{t_{w_{1}}}%
t_{w_{2}}\delta _{r(w_{1}^{-1}w_{2}),\,s(x)}\delta
_{w_{1}^{-1}w_{2}x,\,y}\right. $

$\qquad \qquad \qquad \qquad \qquad \qquad \quad \left. -\underset{%
(y_{1},y_{2})\in Supp(T)^{2},\,y_{1}y_{2}^{-1}\neq \emptyset }{\sum }%
t_{y_{1}}\overline{t_{y_{2}}}\delta _{r(y_{1}y_{2}^{-1}),\,s(x)}\delta
_{y_{1}y_{2}^{-1}x,\,y}\right) $

of (4.1.5) by $\Delta _{xy}.$ By (4.1.3), each summand $\Delta _{xy}$ has
its (kind of) pair $\Delta _{yx},$ in the formula (4.1.5),

$\qquad \Delta _{yx}$ $=$ $r_{y}$ $\overline{r_{x}}$ $\left( \underset{%
(w_{1},w_{2})\in Supp(T)^{2},\,w_{1}^{-1}w_{2}\neq \emptyset }{\sum }%
\overline{t_{w_{2}}}t_{w_{1}}\delta _{r(w_{2}^{-1}w_{1}),\,s(y)}\delta
_{x,\,w_{2}^{-1}w_{1}\,y}\right. $

$\qquad \qquad \qquad \qquad \qquad \qquad \quad \left. -\underset{%
(y_{1},y_{2})\in Supp(T)^{2},\,y_{1}y_{2}^{-1}\neq \emptyset }{\sum }%
t_{y_{2}}\overline{t_{y_{1}}}\delta _{r(y_{2}y_{1}^{-1}),\,s(y)}\delta
_{x,y_{2}y_{1}^{-1}\,y}\right) .$

\strut i.e.,

(4.1.6)

\begin{center}
$<S(T)$ $\xi ,$ $\xi $ $>$ $=$ $\underset{(x,y)\in \Bbb{G}^{2}}{\sum }$ $%
r_{x}$ $\overline{r_{y}}$ $\Delta _{xy}$ $=$ $\underset{(x,y)\in \Bbb{G}^{2}%
}{\sum }$ $r_{y}$ $\overline{r_{x}}$ $\Delta _{yx},$
\end{center}

for all $\xi $ $=$ $\underset{x\in \Bbb{G}}{\sum }$ $r_{x}$ $\xi _{x}$ $\in $
$\mathcal{H}_{G}$ $\subseteq $ $H_{G}.$

\strut The following theorem is the characterization of hyponormal graph
operators.

\begin{theorem}
Let $T$ $=$ $\underset{w\in Supp(T)}{\sum }$ $t_{w}$ $L_{w}$ be a graph
operator in the graph von Neumann algebra $M_{G}$ of $G.$ Then $T$ is
hyponormal, if and only if

(4.1.7)

\begin{center}
$\{r(w)$ $:$ $w$ $\in $ $\Pi _{T^{*}T}\}$ $\supseteq $ $\{r(w)$ $:$ $w$ $\in 
$ $\Pi _{TT^{*}}\}$ in $V(\widehat{G}),$
\end{center}

where

\begin{center}
$\Pi _{T^{*}T}$ $\overset{def}{=}$ $\left( \left( Supp(T)\right) ^{-1}\left(
Supp(T)\right) \right) $ $\setminus $ $\{\emptyset \},$
\end{center}

and

\begin{center}
$\Pi _{TT^{*}}$ $\overset{def}{=}$ $\left( \left( Supp(T)\right) \left(
Supp(T)\right) ^{-1}\right) $ $\setminus $ $\{\emptyset \},$
\end{center}

in $\Bbb{G},$ and the coefficients of $T$ satisfies

(4.1.8)

\begin{center}
$\left( \underset{w_{1}^{-1}w_{2}\in \Pi _{T^{*}T},\,r(w_{1}^{-1}w_{2})=v}{%
\sum }\overline{t_{w_{1}}}t_{w_{2}}\right) $ $\geq $ $\left( \underset{%
y_{1}y_{2}^{-1}\in \Pi _{TT^{*}},\,r(y_{1}y_{2}^{-1})=v}{\sum }t_{w_{1}}%
\overline{t_{w_{2}}}\right) ,$
\end{center}

in $\Bbb{R}$ $\subset $ $\Bbb{C},$ for all $v$ $\in $ $V(\widehat{G}).$
\end{theorem}

\begin{proof}
\strut Let $T$ $=$ $\underset{w\in Supp(T)}{\sum }$ $t_{w}$ $L_{w}$ be a
given graph operator in $M_{G}.$ Then, by (4.1.1), the self-commutator $S(T)$
of $T$ is

\begin{center}
$S(T)$ $=$ $TT^{*}$ $-$ $TT^{*}$ $=$ $\underset{(w_{1},w_{2})}{\sum }$ $%
\overline{t_{w_{1}}}$ $t_{w_{2}}$ $L_{w_{1}^{-1}w_{2}}$ $-$ $\underset{%
(y_{1},y_{2})}{\sum }$ $t_{y_{1}}$ $\overline{t_{y_{2}}}$ $%
L_{y_{1}y_{2}^{-1}},$
\end{center}

identified with

\begin{center}
$\underset{(w_{1},w_{2})\in Supp(T)^{2},\,w_{1}^{-1}w_{2}\neq \emptyset }{%
\sum }$ $\overline{t_{w_{1}}}$ $t_{w_{2}}\left(
L_{w_{1}^{-1}w_{2}}-L_{w_{2}w_{1}^{-1}}\right) ,$
\end{center}

by (4.1.2). Then, for all $\xi $ $=$ $\underset{x\in \Bbb{G}}{\sum }$ $r_{x}$
$\xi _{x}$ $\in $ $\mathcal{H}_{G}$ in $H_{G},$ we can obtain the formula
(4.1.5), which states;

$\qquad <S(T)$ $\xi ,$ $\xi $ $>$

$\qquad \qquad =$ $\underset{(x,y)\in \Bbb{G}^{2}}{\sum }r_{x}\overline{r_{y}%
}$ $\left( \underset{(w_{1},w_{2})\in Supp(T)^{2},\,w_{1}^{-1}w_{2}\neq
\emptyset }{\sum }\overline{t_{w_{1}}}t_{w_{2}}\delta
_{r(w_{1}^{-1}w_{2}),\,s(x)}\delta _{w_{1}^{-1}w_{2}x,\,y}\right. $

$\qquad \qquad \qquad \qquad \qquad \qquad \quad $ $\left. -\underset{%
(y_{1},y_{2})\in Supp(T)^{2},\,y_{1}y_{2}^{-1}\neq \emptyset }{\sum }%
t_{y_{1}}\overline{t_{y_{2}}}\delta _{r(y_{1}y_{2}^{-1}),\,s(x)}\delta
_{y_{1}y_{2}^{-1}x,\,y}\right) ,$

satisfying (4.1.6).

($\Leftarrow $) Consider now that the terms

$\qquad \Delta _{xy}^{o}$ $\overset{def}{=}$ $\frac{1}{r_{x}\overline{r_{y}}}%
\Delta _{xy}$ $=$ $\left( \underset{(w_{1},w_{2})\in
Supp(T)^{2},\,w_{1}^{-1}w_{2}\neq \emptyset }{\sum }\overline{t_{w_{1}}}%
t_{w_{2}}\delta _{r(w_{1}^{-1}w_{2}),\,s(x)}\delta
_{w_{1}^{-1}w_{2}x,\,y}\right. $

$\qquad \qquad \qquad \qquad \qquad \qquad \quad $ $\left. -\underset{%
(y_{1},y_{2})\in Supp(T)^{2},\,y_{1}y_{2}^{-1}\neq \emptyset }{\sum }%
t_{y_{1}}\overline{t_{y_{2}}}\delta _{r(y_{1}y_{2}^{-1}),\,s(x)}\delta
_{y_{1}y_{2}^{-1}x,\,y}\right) $

in (4.1.5). Each $\Delta _{xy}^{o}$ can be re-formulated by

(4.1.9)

$\qquad \Delta _{xy}^{o}$ $=$ $\underset{v\in \mathcal{V}_{T}}{\sum }\left( 
\underset{(w_{1},w_{2})\in Supp(T)^{2},\,w_{1}^{-1}w_{2}\neq \emptyset
,\,r(w_{1}^{-1}w_{2})=v}{\sum }\overline{t_{w_{1}}}t_{w_{2}}\delta
_{v,\,s(x)}\delta _{w_{1}^{-1}w_{2}x,\,y}\right. $

$\qquad \qquad \qquad \qquad \qquad \qquad \quad $ $\left. -\underset{%
(y_{1},y_{2})\in Supp(T)^{2},\,y_{1}y_{2}^{-1}\neq \emptyset
,\,r(y_{1}y_{2}^{-1})=v}{\sum }t_{y_{1}}\overline{t_{y_{2}}}\delta
_{v,\,s(x)}\delta _{y_{1}y_{2}^{-1}x,\,y}\right) ,$

where

\begin{center}
$\mathcal{V}_{T}$ $\overset{def}{=}$ $\{r(w)$ $:$ $w$ $\in $ $Supp(T)\}$ $%
\cup $ $\{s(w)$ $:$ $w$ $\in $ $Supp(T)\},$
\end{center}

in $V(\widehat{G}).$ Let's denote the summand of $\Delta _{xy}^{o}$ by $%
\Delta _{xy}^{o}(v),$ for $v$ $\in $ $\mathcal{V}_{T}.$ i.e.,

(4.1.10)

\begin{center}
$\Delta _{xy}^{o}$ $=$ $\underset{v\in \mathcal{V}_{T}}{\sum }$ $\Delta
_{xy}^{o}$
\end{center}

Thus, if

(4.1.11)

\begin{center}
$\Delta _{xy}^{o}(v)$ $\geq $ $0,$ for all $v$ $\in $ $V(\widehat{G}),$
\end{center}

then we can make $<S(T)\xi ,$ $\xi $ $>$ be positive in $\Bbb{R},$ for an
``arbitrary'' $\xi $ $\in $ $\mathcal{H}_{G},$ and hence the operator $T$ is
hyponormal, by (4.1.6).

So, since the above vector $\xi $ is arbitrary in $\mathcal{H}_{G},$ we can
obtain that: if the set-inclusion (4.1.7) holds, and if the inequality
(4.1.8);

$\qquad \quad \left( \underset{(w_{1},w_{2})\in
Supp(T)^{2},\,w_{1}^{-1}w_{2}\neq \emptyset ,\,r(w_{1}^{-1}w_{2})=v}{\sum }%
\overline{t_{w_{1}}}t_{w_{2}}\right) $

$\qquad \quad \qquad \quad \qquad \quad \geq $ $\left( \underset{%
(y_{1},y_{2})\in Supp(T)^{2},\,y_{1}y_{2}^{-1}\neq \emptyset
,\,r(y_{1}y_{2}^{-1})=v}{\sum }\text{ }t_{y_{1}}\text{ }\overline{t_{y_{2}}}%
\right) ,$

holds in $\Bbb{R}$($\subset $ $\Bbb{C}$), for all $v$ $\in $ $\mathcal{V}%
_{T},$ then $T$ is hyponormal, since

\begin{center}
$<S(T)$ $\eta ,$ $\eta $ $>$ $\geq $ $0,$ for all $\eta $ $\in $ $\mathcal{H}%
_{G}$($\subseteq $ $H_{G}$).
\end{center}

Equivalently, if both (4.1.7) and (4.1.8) hold, then $T$ is hyponormal on $%
H_{G}.$

\strut

($\Rightarrow $) Conversely, let a given graph operator $T$ be hyponormal on 
$H_{G},$ equivalently, the self-commutator $S(T)$ is a positive operator on $%
H_{G}.$ And assume that $S(T)$ does not satisfy either (4.1.7) or (4.1.8).

Suppose first that the condition (4.1.7) does not hold. i.e., assume

(4.1.12)

\begin{center}
$\mathcal{R}_{T^{*}T}$ $=$ $\{r(w)$ $:$ $w$ $\in $ $\Pi _{T^{*}T}\}$ $%
\subset $ $\{r(w)$ $:$ $w$ $\in $ $\Pi _{TT^{*}}\}$ $=$ $\mathcal{R}%
_{TT^{*}}.$
\end{center}

This means that there exists an element $w_{0}$ $\in $ $Supp(T),$ such that

\begin{center}
$r(w_{0})$ $\in $ $\mathcal{R}_{T^{*}T}$ and $s(w_{0})$ $\in $ $\mathcal{R}%
_{TT^{*}},$
\end{center}

satisfying

\begin{center}
$r(w_{0})$ $\neq $ $s(w_{0})$ in $V(\widehat{G}),$
\end{center}

with

\begin{center}
$s(w_{0})$ $\in $ $\mathcal{R}_{TT^{*}}$ $\setminus $ $\mathcal{R}_{T^{*}T}$
($\neq $ $\varnothing $).
\end{center}

Notice here that

$r(w_{1}w_{2})$ $=$ $r(w_{2}),$ and $s(w_{1}w_{2})$ $=$ $s(w_{1}),$

for all $w_{1},$ $w_{2}$ $\in $ $\Bbb{G}.$ So, our condition (4.1.12)
guarantees the existence of such an element $w_{0}$ in $Supp(T).$ Then we
can obtain the summand

\begin{center}
$\left| t_{w_{0}}\right| ^{2}\left( L_{r(w_{0})}-L_{s(w_{0})}\right) $
\end{center}

of $S(T),$ by (4.1.2). Again, by (4.1.12), we have the summand

\begin{center}
$-\left| t_{w_{0}}\right| ^{2}L_{s(w_{0})}$
\end{center}

of $S(T).$ Thus, if we take a vector $\xi _{w_{0}}$ $\in $ $\overline{t_{w}}$
$\in $ $\mathcal{B}_{H_{G}}$ $\subset $ $\mathcal{H}_{G}$ in $H_{G},$ then

$\qquad <S(T)\xi _{w_{0}},$ $\xi _{w_{0}}$ $>$ $=$ $-\left| t_{w_{0}}\right|
^{2}<L_{s(w_{0})}\xi _{w_{0}},$ $\xi _{w_{0}}$ $>$

$\qquad \qquad \qquad =$ $-\left| t_{w_{0}}\right| ^{2}$ $<\xi _{w_{0}},$ $%
\xi _{w_{0}}$ $>$

by (4.1.4)

$\qquad \qquad \qquad =$ $-\left| t_{w_{0}}\right| ^{2}$ $\left\| \xi
_{w_{0}}\right\| ^{2}$ $=$ $-\left| t_{w_{0}}\right| ^{2}$ $<$ $0,$

where

\begin{center}
$\left\| \eta \right\| $ $\overset{def}{=}$ $\sqrt{<\eta ,\text{ }\eta >},$
for all $\eta $ $\in $ $H_{G},$
\end{center}

is the Hilbert space norm on $H_{G}.$ This shows that there exists a vector $%
\xi ^{\prime }$ $\in $ $H_{G},$ such that $<$ $S(T)$ $\xi ^{\prime },$ $\xi
^{\prime }$ $>$ becomes negative in $\Bbb{R}.$ This contradicts our
assumption that $T$ is hyponormal.

Therefore, if $T$ is hyponormal, then the condition (4.1.7) must hold.

Assume now that $T$ is hyponormal, and the inequality (4.1.8) does not hold.
We will assume that (4.1.7) holds true for $T.$ Since (4.1.8) does not hold,
there exists at least one vertex $v_{0}$ such that

(4.1.13)

\begin{center}
$\Delta _{xy}^{o}(v_{0})$ $<$ $0,$
\end{center}

where $\Delta _{xy}^{o}$ and $\Delta _{xy}^{o}(v)$'s are defined in (4.1.10)
and (4.1.11), respectively.

Then we can take a vector

$\qquad \eta $ $=$ $\underset{(w_{1},w_{2})\in
Supp(T)^{2},\,w_{1}^{-1}w_{2}\neq \emptyset ,\,r(w_{1}^{-1}w_{2})=v_{0}}{%
\sum }\left( \xi _{x}+\xi _{w_{1}^{-1}w_{2}x}\right) $

$\qquad \qquad \qquad +$ $\underset{(y_{1},y_{2})\in
Supp(T)^{2},\,y_{1}y_{2}^{-1}\neq \emptyset ,\,r(y_{1}y_{2}^{-1})=v_{0}}{%
\sum }$ $\left( \xi _{y}+\xi _{y_{1}y_{2}^{-1}y}\right) $

in $\mathcal{H}_{G}.$ Then, by (4.1.3)

\begin{center}
$<S(T)\eta ,$ $\eta $ $>$ $=$ $\Delta _{xy}^{o}(v_{0})$ $<$ $0.$
\end{center}

Therefore, it breaks the hyponormality of $T,$ which contradicts our
assumption that $T$ is hyponormal. Thus, the condition (4.1.8) must hold
under the hyponormality of $T.$

As we have seen above, we can conclude that a graph operator $T$ is
hyponormal, if and only if the both conditions (4.1.7), and (4.1.8) hold.
\end{proof}

\strut The above theorem characterize the hyponormality of graph operators
in terms of the admissibility on $\Bbb{G},$ and the analytic data of
coefficients, just like Sections 3.2, and 3.3.

From below, denote

\begin{center}
$r\left( \Pi _{T^{*}T}\right) $ and $r(\Pi _{TT^{*}})$
\end{center}

by

\begin{center}
$\mathcal{R}_{T^{*}T}$ and $\mathcal{R}_{TT^{*}},$
\end{center}

respectively.

The above theorem provides not only the characterization of hyponormal graph
operators but also the very useful process for checking
``non-hyponormality.''

\begin{corollary}
Let $T$ $=$ $\underset{w\in Supp(T)}{\sum }$ $t_{w}L_{w}$ be a graph
operator in $M_{G}.$

(1) If $\mathcal{R}_{T^{*}T}$ $\nsupseteq $ $\mathcal{R}_{TT^{*}},$ then $T$
is not hyponormal.

(2) If there exists a vertex $v_{0}$ $\in $ $\mathcal{V}_{T},$ such that

$\qquad \left( \underset{(w_{1},w_{2})\in Supp(T)^{2},\,w_{1}^{-1}w_{2}\neq
\emptyset ,\,r(w_{1}^{-1}w_{2})=v_{0}}{\sum }\overline{t_{w_{1}}}\text{ }%
t_{w_{2}}\right) $

$\qquad \qquad \qquad \qquad \ngeqq $ $\left( \underset{(y_{1},y_{2})\in
Supp(T)^{2},\,y_{1}y_{2}^{-1}\neq \emptyset ,\,r(y_{1}y_{2}^{-1})=v_{0}}{%
\sum }\text{ }t_{y_{1}}\overline{t_{y_{2}}}\right) ,$

then $T$ is not hyponormal. $\square $
\end{corollary}

In the rest of this section, we will consider several fundamental examples.
These examples will give the concrete understanding for the above theorem;
the characterization of hyponormal graph operators.

\begin{example}
Suppose a graph $G$ contains its subgraph,

\strut

\begin{center}
$_{v_{1}}\bullet \overset{e_{1}}{\underset{e_{2}}{\rightrightarrows }}%
\bullet _{v_{2}},$
\end{center}

\strut

and let $T$ $=$ $t_{e_{1}}$ $L_{e_{1}}$ $+$ $t_{e_{2}}$ $L_{e_{2}}.$ Then
this graph operator $T$ is not hyponormal, since

\begin{center}
$\mathcal{R}_{T^{*}T}$ $=$ $\{r(e_{1}^{-1}e_{1}),$ $r(e_{1}^{-1}e_{2}),$ $%
r(e_{2}^{-1}e_{1}),$ $r(e_{2}^{-1}e_{2})\}$ $=$ $\{v_{2}\},$
\end{center}

and

\begin{center}
$\mathcal{R}_{TT^{*}}$ $=$ $\{r(e_{1}e_{1}^{-1}),$ $r(e_{1}e_{2}^{-1}),$ $%
r(e_{2}e_{1}^{-1}),$ $r(e_{2}e_{2}^{-1})\}$ $=$ $\{v_{1}\}$.
\end{center}

in $V(\widehat{G}).$ So,

\begin{center}
$\mathcal{R}_{T^{*}T}$ $\cap $ $\mathcal{R}_{TT^{*}}$ $=$ $\varnothing ,$
\end{center}

and hence $T$ does not satisfy the condition (4.1.7), stating

\begin{center}
$\mathcal{R}_{T^{*}T}$ $\supseteq $ $\mathcal{R}_{TT^{*}}.$
\end{center}

Therefore, this operator $T$ is not hyponormal.
\end{example}

\begin{example}
\strut Suppose a graph $G$ contains its subgraph,

\strut

\begin{center}
$_{v_{1}}\bullet \overset{e_{1}}{\underset{e_{2}}{\rightrightarrows }}%
\bullet _{v_{2}},$
\end{center}

\strut

and let $T_{1}$ $=$ $t_{e_{1}}L_{e_{1}}$ $+$ $t_{e_{1}^{-1}}$ $%
L_{e_{1}^{-1}} $ $+$ $t_{e_{2}}L_{e_{2}}.$ Then we can have that

\begin{center}
$\mathcal{R}_{T_{1}^{*}T_{1}}$ $=$ $\left\{ 
\begin{array}{c}
r(e_{1}^{-1}e_{1}),\text{ }r(e_{1}^{-1}e_{1}^{-1}),\text{ }%
r(e_{1}^{-1}e_{2}), \\ 
r(e_{1}e_{1}),\text{ }r(e_{1}e_{1}^{-1}),\text{ }r(e_{1}e_{2}), \\ 
r(e_{2}^{-1}e_{1}),\text{ }r(e_{2}^{-1}e_{1}^{-1}),\text{ }r(e_{2}^{-1}e_{2})
\end{array}
\right\} $ $=$ $\{v_{1},$ $v_{2}\},$
\end{center}

and

\begin{center}
$\mathcal{R}_{T_{1}T_{1}^{*}}$ $=$ $\left\{ 
\begin{array}{c}
r(e_{1}e_{1}^{-1}),\text{ }r(e_{1}e_{1}),\text{ }r(e_{1}e_{2}^{-1}), \\ 
r(e_{1}^{-1}e_{1}^{-1}),\text{ }r(e_{1}^{-1}e_{1}),\text{ }%
r(e_{1}^{-1}e_{2}^{-1}), \\ 
r(e_{2}e_{1}^{-1}),\text{ }r(e_{2}e_{1}),\text{ }r(e_{2}e_{2}^{-1})
\end{array}
\right\} $ $=$ $\{v_{1},$ $v_{2}\},$
\end{center}

in $V(\widehat{G}).$ Thus, $T_{1}$ satisfies the condition (4.1.7);

\begin{center}
$\mathcal{R}_{T_{1}^{*}T_{1}}$ $=$ $\mathcal{V}_{T_{1}}$ $=$ $\mathcal{R}%
_{T_{1}T_{1}^{*}},$ and hence $\mathcal{R}_{T_{1}^{*}T_{1}}$ $\supseteq $ $%
\mathcal{R}_{T_{1}T_{1}^{*}}.$
\end{center}

So, $T_{1}$ is hyponormal, if and only if (4.1.8) holds. So, $T_{1}$ is
hyponormal, if and only if, for $v_{1}$ $\in $ $\mathcal{V}_{T},$

(4.1.14)

$\qquad \qquad \left| t_{e_{1}}\right| ^{2}$ $-$ $\left( \left|
t_{e_{1}}\right| ^{2}+t_{e_{1}}\overline{t_{e_{2}}}+t_{e_{2}}\overline{%
t_{e_{1}}}+\left| t_{e_{2}}\right| ^{2}\right) $ $\geq $ $0$

\begin{center}
$\Longleftrightarrow $ $\left| t_{e_{2}}\right| ^{2}$ $\leq $ $-$ $t_{e_{1}}%
\overline{t_{e_{2}}}$ $-$ $t_{e_{2}}\overline{t_{e_{1}}},$
\end{center}

and, for $v_{2}$ $\in $ $\mathcal{V}_{T},$

(4.1.15)

$\qquad \qquad \left( \left| t_{e_{1}}\right| ^{2}+\overline{t_{e_{1}}}%
t_{e_{2}}+\overline{t_{e_{2}}}t_{e_{1}}+\left| t_{e_{2}}\right| ^{2}\right) $
$-$ $\left| t_{e_{1}}\right| ^{2}$ $\geq $ $0$

\begin{center}
$\Longleftrightarrow $ $\left| t_{e_{2}}\right| ^{2}$ $\geq $ $-\overline{%
t_{e_{1}}}$ $t_{e_{2}}$ $-$ $\overline{t_{e_{2}}}$ $t_{e_{1}}.$
\end{center}

If we combine (4.1.14) and (4.1.15), we can obtain that the given operator $%
T_{1}$ is hyponormal, if and only if

(4.1.16)

\begin{center}
$\left| t_{e_{2}}\right| ^{2}$ $=$ $-$ $\overline{t_{e_{1}}}$ $t_{e_{2}}$ $-$
$\overline{t_{e_{2}}}$ $t_{e_{1}}.$
\end{center}

In fact, the readers can easily check that the hyponormality condition
(4.1.16) guarantees the ``normality'' of $T_{1},$ too. i.e., $T_{1}$ is
normal, if and only if (4.1.16) holds (See Section 4.2 below).

\strut

Now, let $T_{2}$ $=$ $t_{e_{1}}L_{e_{1}}$ $+$ $t_{e_{1}^{-1}}$ $%
L_{e_{1}^{-1}}.$ Then we have

$\mathcal{R}_{T_{2}^{*}T_{2}}$ $=$ $\left\{ 
\begin{array}{c}
r(e_{1}^{-1}e_{1}),\text{ }r(e_{1}^{-1}e_{1}^{-1}), \\ 
r(e_{1}e_{1}),\text{ }r(e_{1}e_{1}^{-1})
\end{array}
\right\} $ $=$ $\{v_{1},$ $v_{2}\},$

and

$\mathcal{R}_{T_{2}T_{2}^{*}}$ $=$ $\left\{ 
\begin{array}{c}
r(e_{1}e_{1}^{-1}),\text{ }r(e_{1}e_{1}), \\ 
r(e_{1}^{-1}e_{1}^{-1}),\text{ }r(e_{1}^{-1}e_{1})
\end{array}
\right\} $ $=$ $\{v_{1},$ $v_{2}\},$

and hence the operator $T_{2}$ satisfies (4.1.7). So, $T_{2}$ is hyponormal,
if and only if

(4.1.17)

\begin{center}
$\left| t_{e_{1}}\right| ^{2}$ $\geq $ $\left| t_{e_{1}^{-1}}\right|
^{2}\qquad $(for $v_{1}$),
\end{center}

and

\begin{center}
$\left| t_{e_{1}^{-1}}\right| ^{2}$ $\geq $ $\left| t_{e_{1}}\right| ^{2}$%
\qquad (for $v_{2}$).
\end{center}

Therefore, by (4.1.17), we can conclude that $T_{2}$ is hyponormal, if and
only if

\begin{center}
$\left| t_{e_{1}}\right| ^{2}$ $=$ $\left| t_{e_{1}^{-1}}\right| ^{2}$ in $%
\Bbb{C}.$
\end{center}

This example also shows that the hyponormality of $T_{2}$ is equivalent to
the normality of $T_{2}$ (See Section 4.2 below).
\end{example}

\begin{example}
\strut Let a graph $G$ contains the following subgraph,

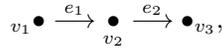
\strut

\begin{center}
$_{v_{1}}\bullet \overset{e_{1}}{\longrightarrow }\underset{v_{2}}{\bullet }%
\overset{e_{2}}{\longrightarrow }\bullet _{v_{3}},$
\end{center}

\strut

and let $T$ $=$ $t_{1}L_{e_{1}}$ $+$ $t_{2}L_{e_{2}}$ $+$ $%
t_{v_{3}}L_{v_{3}}.$ Then we can have that

\begin{center}
$\mathcal{R}_{T^{*}T}$ $=$ $\left\{ 
\begin{array}{c}
r(e_{1}^{-1}e_{1}),\text{ }r(e_{1}^{-1}e_{2}),\text{ }r(e_{1}^{-1}v_{3}), \\ 
r(e_{2}^{-1}e_{1}),\text{ }r(e_{2}^{-1}e_{2}),\text{ }r(e_{2}^{-1}v_{3}), \\ 
r(v_{3}e_{1}),\text{ }r(v_{3}e_{2}),\text{ }r(v_{3}v_{3})
\end{array}
\right\} $ $=$ $\{v_{1},$ $v_{3}\},$
\end{center}

and

\begin{center}
$\mathcal{R}_{TT^{*}}$ $=$ $\left\{ 
\begin{array}{c}
r(e_{1}e_{1}^{-1}),\text{ }r(e_{1}e_{2}^{-1}),\text{ }r(e_{1}v_{3}), \\ 
r(e_{2}e_{1}^{-1}),\text{ }r(e_{2}e_{2}^{-1}),\text{ }r(e_{2}v_{3}), \\ 
r(v_{3}e_{1}^{-1}),\text{ }r(v_{3}e_{2}^{-1}),\text{ }r(v_{3}v_{3})
\end{array}
\right\} $ $=$ $\{v_{1},$ $v_{2},$ $v_{3}\}.$
\end{center}

So, the operator $T$ does not satisfy the condition (4.1.7) for the
hyponormality of $T$, and hence $T$ is not hyponormal.
\end{example}

\begin{example}
\strut Assume that a graph $G$ contains a subgraph,

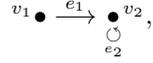
\strut

\begin{center}
$^{v_{1}}\bullet \overset{e_{1}}{\longrightarrow }\underset{\underset{e_{2}}{%
\circlearrowleft }}{\bullet }^{v_{2}},$
\end{center}

\strut

and let $T$ $=$ $t_{v_{1}}L_{v_{1}}$ $+$ $t_{e_{1}}$ $L_{e_{1}}$ $+$ $%
t_{e_{2}^{-1}}L_{e_{2}^{-1}}.$ Then we can have that

\begin{center}
$\mathcal{R}_{T^{*}T}$ $=$ $\left\{ 
\begin{array}{c}
r(v_{1}v_{1}),\text{ }r(v_{1}e_{1}),\text{ }r(v_{1}e_{2}^{-1}), \\ 
r(e_{1}^{-1}v_{1}),\text{ }r(e_{1}^{-1}e_{1}),\text{ }r(e_{1}^{-1}e_{2}), \\ 
r(e_{2}v_{1}),\text{ }r(e_{2}e_{1}),\text{ }r(e_{2}e_{2}^{-1})
\end{array}
\right\} $ $=$ $\{v_{1},$ $v_{2}\},$
\end{center}

and

\begin{center}
$\mathcal{R}_{TT^{*}}$ $=$ $\left\{ 
\begin{array}{c}
r(v_{1}v_{1}),\text{ }r(v_{1}e_{1}^{-1}),\text{ }r(v_{1}e_{2}), \\ 
r(e_{1}v_{1}),\text{ }r(e_{1}e_{1}^{-1}),\text{ }r(e_{1}e_{2}), \\ 
r(e_{2}^{-1}v_{1}),\text{ }r(e_{2}^{-1}e_{1}^{-1}),\text{ }r(e_{2}^{-1}e_{2})
\end{array}
\right\} $ $=$ $\{v_{1},$ $v_{2}\}.$
\end{center}

So, $\mathcal{R}_{T^{*}T}$ $=$ $\mathcal{R}_{TT^{*}},$ and hence $T$
satisfies the condition (4.1.7). So, to make $T$ be hyponormal, the
coefficients of $T$ must satisfy the condition (4.1.8). Thus we can conclude
that $T$ is hyponormal, if and only if

(4.1.18)

\begin{center}
$\left( \left| t_{v_{1}}\right| ^{2}+\overline{t_{e_{1}}}t_{v_{1}}\right) $ $%
\geq $ $\left( \left| t_{v_{1}}\right| ^{2}+\left| t_{e_{1}}\right|
^{2}+t_{e_{2}^{-1}}\overline{t_{e_{1}}}\right) \quad $(for $v_{1}$)
\end{center}

and

\begin{center}
$\left( \overline{t_{v_{1}}}\text{ }t_{e_{1}}+\left| t_{e_{1}}\right|
^{2}+\left| t_{e_{2}^{-1}}\right| ^{2}\right) $ $\geq $ $\left( t_{e_{1}}%
\overline{t_{e_{2}^{-1}}}+\left| t_{e_{2}^{-1}}\right| ^{2}\right) $\quad
(for $v_{2}$).
\end{center}

The above condition (4.1.18) can be rewritten by

(4.1.19)

\begin{center}
$\overline{t_{e_{1}}}$ $t_{v_{1}}$ $-$ $t_{e_{2}^{-1}}\overline{t_{e_{1}}}$ $%
\geq $ $\left| t_{e_{1}}\right| ^{2},$
\end{center}

and

\begin{center}
$\left| t_{e_{1}}\right| ^{2}$ $\geq $ $t_{e_{1}}\overline{t_{e_{2}^{-1}}}$ $%
-$ $\overline{t_{v_{1}}}$ $t_{e_{1}},$
\end{center}

respectively. Therefore, the given graph operator $T$ is hyponormal, if and
only if

\begin{center}
$\overline{t_{e_{1}}}$ $t_{v_{1}}$ $-$ $t_{e_{2}^{-1}}\overline{t_{e_{1}}}$ $%
\geq $ $\left| t_{e_{1}}\right| ^{2}$ $\geq $ $t_{e_{1}}\overline{%
t_{e_{2}^{-1}}}$ $-$ $\overline{t_{v_{1}}}$ $t_{e_{1}},$

if and only if

$\left| t_{e_{1}}\right| ^{2}$ $\leq $ $\left| \overline{t_{e_{1}}}%
t_{v_{1}}-t_{e_{2}^{-1}}\overline{t_{e_{1}}}\right| $.
\end{center}
\end{example}

\begin{example}
\strut Suppose a graph $G$ contains its subgraph,

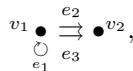
\strut \strut

\begin{center}
$^{v_{1}}\underset{\underset{e_{1}}{\circlearrowright }}{\bullet }\overset{%
e_{2}}{\underset{e_{3}}{\rightrightarrows }}\bullet ^{v_{2}},$

\strut
\end{center}

and let $T$ $=$ $t_{e_{1}}L_{e_{1}}$ $+$ $t_{e_{2}}$ $L_{e_{2}}$ $+$ $%
t_{e_{3}}$ $L_{e_{3}}.$ Then we can have that

\begin{center}
$\mathcal{R}_{T^{*}T}$ $=$ $\left\{ 
\begin{array}{c}
r(e_{1}^{-1}e_{1}),\text{ }r(e_{1}^{-1}e_{2}),\text{ }r(e_{1}^{-1}e_{3}), \\ 
r(e_{2}^{-1}e_{1}),\text{ }r(e_{2}^{-1}e_{2}),\text{ }r(e_{2}^{-1}e_{3}), \\ 
r(e_{3}^{-1}e_{1}),\text{ }r(e_{3}^{-1}e_{2}),\text{ }r(e_{3}^{-1}e_{3})
\end{array}
\right\} $ $=$ $\{v_{1},$ $v_{2}\},$
\end{center}

and

\begin{center}
$\mathcal{R}_{TT^{*}}$ $=$ $\left\{ 
\begin{array}{c}
r(e_{1}e_{1}^{-1}),\text{ }r(e_{1}e_{2}^{-1}),\text{ }r(e_{1}e_{3}^{-1}), \\ 
r(e_{2}e_{1}^{-1}),\text{ }r(e_{2}e_{2}^{-1}),\text{ }r(e_{2}e_{3}^{-1}), \\ 
r(e_{3}e_{1}^{-1}),\text{ }r(e_{3}e_{2}^{-1}),\text{ }r(e_{3}e_{3}^{-1})
\end{array}
\right\} $ $=$ $\{v_{1}\}.$
\end{center}

So, the operator $T$ satisfies the condition (4.1.7) i.e.,

\begin{center}
$\mathcal{R}_{T^{*}T}$ $\supset $ $\mathcal{R}_{TT^{*}}.$
\end{center}

Thus, we can obtain that $T$ is hyponormal, if and only if

(4.1.21)

\begin{center}
$\left( \left| t_{e_{1}}\right| ^{2}+\overline{t_{e_{2}}}\text{ }t_{e_{1}}+%
\overline{t_{e_{3}}}\text{ }t_{e_{1}}\right) $ $\geq $ $\left( \left|
t_{e_{1}}\right| ^{2}+\left| t_{e_{2}}\right| ^{2}+t_{e_{2}}\overline{%
t_{e_{3}}}+t_{e_{3}}\overline{t_{e_{2}}}+\left| t_{e_{3}}\right| ^{2}\right) 
$
\end{center}

(for $v_{1}$), and

\begin{center}
$\left( \overline{t_{e_{1}}}t_{e_{2}}+\overline{t_{e_{1}}}t_{e_{3}}+\left|
t_{e_{2}}\right| ^{2}+\overline{t_{e_{2}}}\text{ }t_{e_{3}}+\overline{%
t_{e_{3}}}t_{e_{2}}+\left| t_{e_{3}}\right| ^{2}\right) $ $\geq $ $0$
\end{center}

(for $v_{2}$).
\end{example}

\subsection{Normality}

In this section, we will consider the normality of graph operators. In
Section 4.1, we studied the hyponormality of graph operators in terms of
combinatorial information of given graphs, and certain analytic data of
coefficients of operators. Throughout this section, we will use the same
notations used in Section 4.1.

Thanks to the hyponormality characterization ((4.1.7) and (4.1.8)) of graph
operators, we can obtain the following normality characterization of graph
operators.

\begin{theorem}
Let $T$ $=$ $\underset{w\in Supp(T)}{\sum }$ $t_{w}$ $L_{w}$ be a graph
operator in the graph von Neumann algebra $M_{G}$ of a connected graph $G.$
Then $T$ is normal, if and only if

(4.2.1)

\begin{center}
$\mathcal{R}_{T^{*}T}$ $=$ $\mathcal{R}_{TT^{*}},$
\end{center}

and

(4.2.2)

$\qquad \qquad \left( \underset{(w_{1},w_{2})\in
Supp(T)^{2},\,w_{1}^{-1}w_{2}\neq \emptyset ,\,r(w_{1}^{-1}w_{2})=v}{\sum }%
\text{ }\overline{t_{w_{1}}}\text{ }t_{w_{2}}\right) $

\strut

$\qquad \qquad \qquad \qquad \qquad \geq $ $\left( \underset{%
(y_{1},y_{2})\in Supp(T)^{2},\,y_{1}y_{2}^{-1}\neq \emptyset
,\,r(y_{1}y_{2}^{-1})=v}{\sum }\text{ }t_{y_{1}}\text{ }\overline{t_{y_{2}}}%
\right) .$
\end{theorem}

\begin{proof}
By definition, a graph operator $T$ is normal on the graph Hilbert space $%
H_{G},$ if and only if $T^{*}T$ $=$ $TT^{*}$ on $H_{G}.$ In other words, $T$
is normal, if and only if both $T$ and $T^{*}$ are hyponormal. Thus, $T$ is
normal, if and only if the self-commutator $S(T)$ is identical to the zero
element $0_{M_{G}}$ (which is identified with the zero operator $0_{H_{G}}$
on $H_{G}$), if and only if

\begin{center}
$<S(T)\xi ,$ $\xi >$ $=$ $0,$ for all $\xi $ $\in $ $H_{G}.$
\end{center}

Therefore, by the little modification of the proof of Theorem 4.5, we can
conclude that $T$ is normal, if and only if the combinatorial condition
(4.2.1) and the analytic condition (4.2.2) hold.\strut
\end{proof}

\section{Operators in Free Group Factors}

In this section, we consider applications of operator-theoretic properties
of graph operators. We will characterize the self-adjointness, the
hyponormality, the normality, and the unitary property of finitely supported
operators in the free group factor $L(F_{N}),$ generated by the free group $%
F_{N}$ with $N$-generators, for $N$ $\in $ $\Bbb{N}.$

In operator algebra, the study of free group factors $L(F_{N})$ is very
important (e.g., See [11]). Also, the study of elements of $L(F_{N})$ is
interesting, since they are (possibly, the infinite or the limit of) linear
combinations of unitary operators (e.g., See [3], [4], [6], and [7]). The
following theorem provides the key motivation of our applications.

\begin{theorem}
\strut (Also, see [4]) The free group factor $L(F_{N})$ is $*$-isomorphic to
the graph von Neumann algebra $M_{O_{N}}$ of the one-vertex-$N$-loop-edge
graph $O_{N},$ for all $N$ $\in $ $\Bbb{N}.$
\end{theorem}

\begin{proof}
\strut Let $O_{N}$ be the one-vertex-$N$-loop-edge graph and let $\Bbb{O}%
_{N} $ be the graph groupoid of $O_{N}.$ Since $O_{N}$ has only one vertex,
say $v_{O},$ the graph groupoid $\Bbb{O}_{N}$ is in fact a group (See
Section 2.2). Indeed, the graph groupoid $\Bbb{O}_{N}$ is a (categorial)
groupoid (in the sense of Section 2.2) with its base, consisting of only one
element $v_{O}.$ Thus $\Bbb{O}_{N}$ is a group. Moreover, this group $\Bbb{O}%
_{N}$ has $N$-generators contained in the edge set

\begin{center}
$E(O_{N})$ $=$ $\{e_{1},$ ..., $e_{N}\}$
\end{center}

of $O_{N}.$ So, we can define a morphism $g$ $:$ $\Bbb{O}_{N}$ $\rightarrow $
$F_{N}$ by a map satisfying

\begin{center}
$g$ $:$ $e_{j}$ $\in $ $E(O_{N})$ $\mapsto $ $u_{j}$ $\in $ $X_{F_{N}},$
\end{center}

for all $j$ $=$ $1,$ ..., $N$ (by the possible rearrangement), where

\begin{center}
$X_{F_{N}}$ $=$ $\{u_{1},$ ..., $u_{N}\}$
\end{center}

is the generator set of the free group $F_{N}$ $=$ $<X_{F_{N}}>.$

Then this morphism satisfies that

$g(x_{i_{1}}$ ... $x_{i_{n}})$ $=$ $q_{i_{1}}$ ... $q_{i_{n}}$ in $F_{N},$

for all $x_{i_{1}},$ ..., $x_{i_{n}}$ $\in $ $E(\widehat{O_{N}}),$ for $n$ $%
\in $ $\Bbb{N},$ such that

\begin{center}
$x_{i_{j}}$ $=$ $\left\{ 
\begin{array}{ll}
e_{i_{j}} & \text{if }x_{i_{j}}\in E(O_{N}) \\ 
e_{i_{j}}^{-1} & \text{if }x_{i_{j}}\in E(O_{N}^{-1}),
\end{array}
\right. $
\end{center}

where $\widehat{O_{N}}$ is the shadowed graph of $O_{N},$ and where

\begin{center}
$q_{i_{j}}$ $=$ $\left\{ 
\begin{array}{ll}
u_{i_{j}} & \text{if }q_{i_{j}}\in X_{F_{N}} \\ 
u_{i_{j}}^{-1} & \text{if }q_{i_{j}}\in X_{F_{N}}^{-1},
\end{array}
\right. $
\end{center}

for all $j$ $=$ $1,$ ..., $n.$ (Remark that the graph groupoid $\Bbb{O}_{N}$
is generated by $E(\widehat{O_{N}}),$ as a groupoid, and hence the group $%
\Bbb{O}_{N}$ is generated by $E(O_{N}).$)

Therefore, the morphism $g$ is a group-homomorphism. Since $g$ is preserving
generators, it is bijective. So, the morphism $g$ is a group-isomorphism,
and hence $\Bbb{O}_{N}$ and $F_{N}$ are group-isomorphic.

Let $(H_{O_{N}},$ $L)$ be the canonical representation of $\Bbb{O}_{N},$ and
let $(H_{F_{N}},$ $\lambda )$ be the left regular unitary representation of $%
F_{N},$ where $H_{F_{N}}$ $=$ $l^{2}(F_{N})$ is the group Hilbert space of $%
F_{N}.$ By the existence of the group-isomorphism $g$ of $\Bbb{O}_{N}$ and $%
F_{N},$ the Hilbert spaces $H_{O_{N}}$ and $H_{F_{N}}$ are Hilbert-space
isomorphic. Indeed, there exists a linear map

\begin{center}
$\Phi $ $:$ $H_{O_{N}}$ $\rightarrow $ $H_{F_{N}}$
\end{center}

satisfying that

\begin{center}
$\Phi \left( \underset{w\in \Bbb{O}_{N}}{\sum }\text{ }t_{w}\xi _{w}\right) $
$\overset{def}{=}$ $\underset{g(w)\in g(\Bbb{O}_{N})=F_{N}}{\sum }$ $t_{w}$ $%
\xi _{g(w)},$
\end{center}

in $H_{F_{N}},$ for all $\underset{w\in \Bbb{O}_{N}}{\sum }$ $t_{w}$ $\xi
_{w}.$ It is easy to check that this linear map $\Phi $ is bounded and
bijective. i.e., $\Phi $ is a Hilbert-space isomorphism, and hence $%
H_{O_{N}} $ and $H_{F_{N}}$ are Hilbert-space isomorphic.

By the existence of $\Phi $ and $g,$ we can obtain the commuting diagram,

\begin{center}
$
\begin{array}{lll}
H_{O_{N}} & \overset{\Phi }{\longrightarrow } & H_{F_{N}} \\ 
\,\downarrow _{L} &  & \,\downarrow _{\lambda } \\ 
H_{O_{N}} & \underset{\Phi }{\longrightarrow } & H_{F_{N}}.
\end{array}
$
\end{center}

This shows that the group actions $L$ of $\Bbb{O}_{N}$ and $\lambda $ of $%
F_{N}$ are equivalent. i.e., the representations $(H_{O_{N}},$ $L)$ of $\Bbb{%
O}_{N}$ and $(H_{F_{N}},$ $\lambda )$ of $F_{N}$ are equivalent.

Therefore, the group von Neumann algebras $vN(L(\Bbb{O}_{N}))$ and $%
vN(\lambda (F_{N}))$ are $*$-isomorphic from each other in $B(\mathcal{H}),$
where

\begin{center}
$H_{O_{N}}$ $\overset{\text{Hilbert}}{=}$ $\mathcal{H}$ $\overset{\text{%
Hilbert}}{=}$ $H_{F_{N}},$
\end{center}

where $\overset{\text{Hilbert}}{=}$ means ``being Hilbert-space
isomorphic.'' i.e., the graph von Neumann algebra $M_{O_{N}}$ and the group
von Neumann algebra $L(F_{N})$ are $*$-isomorphic.
\end{proof}

\strut The above theorem shows that the study of $L(F_{N})$ is to study $%
M_{O_{N}}.$ So, to study finitely supported operators of $L(F_{N}),$ we will
study the graph operators in $M_{O_{N}}.$

\strut By Section 3.2, we can obtain the following self-adjointness
characterization on $M_{O_{N}}.$

\begin{proposition}
Let $T$ $=$ $\underset{w\in Supp(T)}{\sum }$ $t_{w}L_{w}$ be a graph
operator in $M_{O_{N}}.$ Then $T$ is self-adjoint, if and only if there
exists a subset $Y$ of

\begin{center}
$Supp(T)$ $\cap $ $FP_{r}(\widehat{O_{N}}),$
\end{center}

such that

(5.1)

\begin{center}
$Supp(T)$ $=$ $\left\{ 
\begin{array}{ll}
\{v_{O}\}\sqcup Y\sqcup Y^{-1} & \text{if }v_{O}\in Supp(T) \\ 
Y\sqcup Y^{-1} & \text{otherwise,}
\end{array}
\right. $
\end{center}

and

\begin{center}
$t_{v_{O}}$ $\in $ $\Bbb{R},$ and $t_{y}$ $=$ $\overline{t_{y^{-1}}}$ in $%
\Bbb{C},$ for all $y$ $\in $ $Y,$
\end{center}

$\square $
\end{proposition}

\strut The proof is done by Section 3.2. By the above proposition, we obtain
the self-adjointness characterization of finitely supported elements in the
free group factor $L(F_{N}).$

\begin{corollary}
Let $T$ $=$ $t_{j_{k}}$ $u_{g_{j_{k}}^{-1}}$ $+$ ... $+$ $%
t_{j_{1}}u_{g_{j_{1}}^{-1}}+t_{0}u_{e}+t_{i_{1}}u_{g_{i_{1}}}+...+t_{i_{n}}u_{g_{i_{n}}} 
$ be an element of $L(F_{N}),$ where $u_{g}$ $\overset{def}{=}$ $\lambda
(g), $ for all $g$ $\in $ $F_{N},$ and $e$ is the group-identity of $F_{N}.$
Then $T$ is self-adjoint, if and only if

(5.2)

\begin{center}
$k$ $=$ $n$ in $\Bbb{N},$
\end{center}

and

\begin{center}
$t_{0}$ $\in $ $\Bbb{R},$ and $t_{i_{p}}$ $=$ $\overline{t_{j_{p}}},$ for
all $p$ $=$ $1,$ ..., $n$ $=$ $k.$
\end{center}

$\square $
\end{corollary}

\strut Now, let's consider the hyponormality.

\begin{proposition}
Let $T$ $=$ $\underset{w\in Supp(T)}{\sum }$ $t_{w}L_{w}$ be a graph
operator in $M_{O_{N}}.$ Then $T$ is hyponormal, if and only if

(5.3)

\begin{center}
$\underset{(w_{1},w_{2})\in Supp(T)^{2}}{\sum }$ $\left( \overline{t_{w_{1}}}%
t_{w_{2}}-t_{w_{1}}\overline{t_{w_{2}}}\right) $ $\geq $ $0.$
\end{center}
\end{proposition}

\begin{proof}
\strut By Section 4.1, we can have, in general, that a graph operator $T$ is
hyponormal, if and only if

(5.4)

\begin{center}
$\mathcal{R}_{T^{*}T}$ $=$ $r\left( \Pi _{T^{*}T}\right) $ $\supseteq $ $%
\mathcal{R}_{TT^{*}}$ $=$ $r\left( \Pi _{TT^{*}}\right) ,$
\end{center}

and

(5.5)

$\qquad \left( \underset{(w_{1},w_{2})\in Supp(T)^{2},\,w_{1}^{-1}w_{2}\neq
\emptyset ,\,r(w_{1}^{-1}w_{2})=v}{\sum }\overline{t_{w_{1}}}%
t_{w_{2}}\right) $

\begin{center}
$\geq $ $\left( \underset{(y_{1},y_{2})\in Supp(T)^{2},\,y_{1}y_{2}^{-1}\neq
\emptyset ,\,r(y_{1}y_{2}^{-1})=v}{\sum }\text{ }t_{y_{1}}\text{ }\overline{%
t_{y_{2}}}\right) ,$
\end{center}

for all $v$ $\in $ $V(\widehat{O_{N}}),$ by (4.1.7), and (4.1.8). However,
the fixed graph $O_{N}$ has only one vertex $v_{O},$ and all elements of $%
\Bbb{O}_{N}$ are admissible from each other via $v_{O}$ (equivalently, $\Bbb{%
O}_{N}$ is a group). Therefore, the condition (5.4) automatically hold true,
and the inequality (5.5) can be simply re-written by

(5.6)

\begin{center}
$\underset{(w_{1},w_{2})\in Supp(T)^{2}}{\sum }$ $\overline{t_{w_{1}}}$ $%
t_{w_{2}}$ $\geq $ $\underset{(y_{1},y_{2})\in Supp(T)^{2}}{\sum }$ $%
t_{y_{1}}$ $\overline{t_{y_{2}}}.$
\end{center}

Therefore, the operator $T$ is hyponormal, if and only if (5.3) holds true.
\end{proof}

\strut By the above proposition, we can obtain that:

\begin{corollary}
Let $T$ $=$ $\underset{g\in Supp(T)}{\sum }$ $t_{g}$ $u_{g}$ be a finitely
supported element of $L(F_{N}).$ Then $T$ is hyponormal, if and only if

(5.7)

\begin{center}
$\underset{(g_{1},g_{2})\in Supp(T)^{2}}{\sum }\left( \overline{t_{g_{1}}}%
\text{ }t_{g_{2}}-t_{g_{1}}\overline{t_{g_{2}}}\right) $ $\geq $ $0.$
\end{center}

$\square $
\end{corollary}

\strut By the hyponormality characterization (5.7) and by Section 4.2, we
can obtain the following corollary, too.

\begin{corollary}
Let $T$ $=$ $\underset{g\in Supp(T)}{\sum }$ $t_{g}u_{g}$ be a finitely
supported element of $L(F_{N}).$ Then $T$ is normal, if and only if

(5.8)

\begin{center}
$\underset{(g_{1},g_{2})\in Supp(T)^{2}}{\sum }\left( \overline{t_{g_{1}}}%
\text{ }t_{g_{2}}-t_{g_{1}}\overline{t_{g_{2}}}\right) $ $=$ $0.$
\end{center}

$\square $
\end{corollary}

\strut Finally, let's consider the unitary property of finitely supported
elements of $L(F_{N}).$ In Section 3.3, we obtain that: a graph operator $T$
of the graph von Neumann algebra $M_{O_{N}}$ of the one-vertex-$N$-loop-edge
graph $O_{N}$ is unitary, if and only if

\begin{center}
$\left( Supp(T)\right) ^{-1}\left( Supp(T)\right) $ $=$ $\{v_{O}\},$
\end{center}

and

\begin{center}
$\underset{(w_{1},w_{2})\in Supp(T)^{2}}{\sum }$ $\overline{t_{w_{1}}}$ $%
t_{w_{2}}$ $=$ $1,$ in $\Bbb{C},$
\end{center}

where $v_{O}$ is the unique vertex of $O_{N},$ for $N$ $\in $ $\Bbb{N}.$
Therefore, we can obtain that:

\begin{proposition}
Let $T$ $=$ $\underset{g\in Supp(T)}{\sum }$ $t_{g}u_{g}$ be a finitely
supported element of $L(F_{N}).$ Then $T$ is unitary, if and only if

(5.9)

\begin{center}
$\left( Supp(T)\right) ^{-1}\left( Supp(T)\right) $ $=$ $\{e_{F_{N}}\},$
\end{center}

and

(5.10)

\begin{center}
$\underset{(g_{1},g_{2})\in Supp(T)^{2}}{\sum }$ $\overline{t_{g_{1}}}$ $%
t_{g_{2}}$ $=$ $1,$ in $\Bbb{C}.$
\end{center}

$\square $
\end{proposition}

\strut

\strut

\strut

\textbf{Appendix A. Categorial Groupoids and Groupoid Actions\strut \strut }

\strut

We say an algebraic structure $(\mathcal{X},$ $\mathcal{Y},$ $s,$ $r)$ is a 
\emph{(categorial) groupoid}, if it satisfies that: (i) $\mathcal{Y}$ $%
\subset $ $\mathcal{X},$ (ii) for all $x_{1},$ $x_{2}$ $\in $ $\mathcal{X},$
there exists a partially-defined binary operation $(x_{1},$ $x_{2})$ $%
\mapsto $ $x_{1}$ $x_{2},$ for all $x_{1},$ $x_{2}$ $\in $ $\mathcal{X},$
depending on the \emph{source map} $s$ and the\emph{\ range map} $r$
satisfying the followings;\strut

(ii-1) $x_{1}$ $x_{2}$ is well-determined, whenever $r(x_{1})$ $=$ $s(x_{2})$
and in this case,

\begin{center}
$s(x_{1}$ $x_{2})$ $=$ $s(x_{1})$ and $r(x_{1}$ $x_{2})$ $=$ $r(x_{2}),$
\end{center}

\qquad \ for $x_{1},$ $x_{2}$ $\in $ $\mathcal{X},$\strut

(ii-2) $(x_{1}$ $x_{2})$ $x_{3}$ $=$ $x_{1}$ $(x_{2}$ $x_{3})$, if they are
well-determined in the sense of (ii-1), for $x_{1},$ $x_{2},$ $x_{3}$ $\in $ 
$\mathcal{X},$\strut

(ii-3) if $x$ $\in $ $\mathcal{X},$ then there exist $y,$ $y^{\prime }$ $\in 
$ $\mathcal{Y}$ such that $s(x)$ $=$ $y$ and $r(x)$ $=$ $y^{\prime },$
satisfying $x$ $=$ $y$ $x$ $y^{\prime }$ (Here, the elements $y$ and $%
y^{\prime }$ are not necessarily distinct),\strut \strut

(ii-4) if $x$ $\in $ $\mathcal{X},$ then there exists a unique element $%
x^{-1}$ for $x$ satisfying

\begin{center}
$x$ $x^{-1}$ $=$ $s(x)$ and $x^{-1}$ $x$ $=$ $r(x).$\strut
\end{center}

The subset $\mathcal{Y}$ of a groupoid $\mathcal{X}$ is said to be the \emph{%
base of} $\mathcal{X}$.

Thus, every group $\Gamma $ is a groupoid $\Gamma $ $=$ $(\Gamma ,$ $%
\{e_{\Gamma }\}$ $s,$ $r)$ (and hence $s$ $=$ $r$ on $\Gamma $), where $%
e_{\Gamma }$ is the group-identity of $\Gamma .$ Conversely, every groupoid
with its base having the cardinality 1 is a group.

Remark that we can naturally assume that there exists the \emph{empty element%
} $\emptyset $ in a groupoid $\mathcal{X}.$ The empty element $\emptyset $
means the products $x_{1}$ $x_{2}$ are not well-defined, for some $x_{1},$ $%
x_{2}$ $\in $ $\mathcal{X}.$ Notice that if $\left| \mathcal{Y}\right| $ $=$ 
$1$ (equivalently, if $\mathcal{X}$ is a group), then the empty word $%
\emptyset $ is not contained in the groupoid $\mathcal{X}.$ However, in
general, whenever $\left| \mathcal{Y}\right| $ $\geq $ $2,$ a groupoid $%
\mathcal{X}$ always contain the empty word. So, if there is no confusion,
the existence of the empty element $\emptyset $ is automatically assumed,
whenever the base $\mathcal{Y}$ of $\mathcal{X}$ contains more than one
element. Under this setting, the partially-defined binary operation on $%
\mathcal{X}$ is well-defined on $\mathcal{X}$ (more precisely, on $\mathcal{X%
}$ $\cup $ $\{\emptyset \},$ which is identified with $\mathcal{X},$
whenever $\left| \mathcal{Y}\right| $ $\geq $ $2$).

It is easy to check that our graph groupoid $\Bbb{G}$ of a countable
directed graph $G$ is indeed a groupoid with its base $V(\widehat{G}).$
i.e., the graph groupoid $\Bbb{G}$ of a graph $G$ is a groupoid

\begin{center}
$\Bbb{G}$ $=$ $(\Bbb{G},$ $V(\widehat{G}),$ $s$, $r)$,
\end{center}

satisfying

\begin{center}
$s(w)$ $=$ $s(v$ $w)$ $=$ $v$ and $r(w)$ $=$ $r(w$ $v^{\prime })$ $=$ $%
v^{\prime },$
\end{center}

for all $w$ $=$ $v$ $w$ $v^{\prime }$ $\in $ $\Bbb{G}$ with $v,$ $v^{\prime
} $ $\in $ $V(\widehat{G}).$ i.e., the vertex set $V(\widehat{G})$ $=$ $V(G)$
is the base of $\Bbb{G}.$\strut

Let $\mathcal{X}_{k}$ $=$ $(\mathcal{X}_{k},$ $\mathcal{Y}_{k},$ $s_{k},$ $%
r_{k})$ be groupoids, for $k$ $=$ $1,$ $2.$ We say that a map $f$ $:$ $%
\mathcal{X}_{1}$ $\rightarrow $ $\mathcal{X}_{2}$ is a \emph{%
groupoid-morphism}, if

(i)$\ \ \ f$ is a function,

(ii)$\ \ f(\mathcal{Y}_{1})$ $\subseteq $ $\mathcal{Y}_{2},$

(iii) $s_{2}\left( f(x)\right) $ $=$ $f\left( s_{1}(x)\right) $ in $\mathcal{%
X}_{2},$ for all $x$ $\in $ $\mathcal{X}_{1}$, and

(iv) $r_{2}\left( f(x)\right) $ $=$ $f\left( r_{1}(x)\right) $ in $\mathcal{X%
}_{2},$ for all $x$ $\in $ $\mathcal{X}_{1}.$

Equivalently, $f$ is a groupoid-morphism, if and only if (i)$^{\prime }$ $f$
is a function, (ii)$^{\prime }$ $f$ satisfies

\begin{center}
$f(x_{1}x_{2})$ $=$ $f(x_{1})$ $f(x_{2})$ in $\mathcal{X}_{2},$
\end{center}

for all $x_{1},$ $x_{2}$ $\in $ $\mathcal{X}_{1}.$

If a groupoid-morphism $f$ is bijective, then we say that $f$ is a \emph{%
groupoid-isomorphism}, and the groupoids $\mathcal{X}_{1}$ and $\mathcal{X}%
_{2}$ are said to be \emph{groupoid-isomorphic}.\strut

Notice that, if two countable directed graphs $G_{1}$ and $G_{2}$ are \emph{%
graph-isomorphic}, via a graph-isomorphism $g$ $:$ $G_{1}$ $\rightarrow $ $%
G_{2},$ in the sense that:

(i)$\ \ \ g$ is bijective from $V(G_{1})$ onto $V(G_{2}),$

(ii)$\ \ g$ is bijective from $E(G_{1})$ onto $E(G_{2}),$

(iii) $g(e)$ $=$ $g(v_{1}$ $e$ $v_{2})$ $=$ $g(v_{1})$ $g(e)$ $g(v_{2})$ in $%
E(G_{2}),$

for all $e$ $=$ $v_{1}$ $e$ $v_{2}$ $\in $ $E(G_{1}),$ with $v_{1},$ $v_{2}$ 
$\in $ $V(G_{1}),$ then the graph groupoids $\Bbb{G}_{1}$ and $\Bbb{G}_{2}$
are groupoid-isomorphic. More generally, if two graphs $G_{1}$ and $G_{2}$
have graph-isomorphic shadowed graphs $\widehat{G_{1}}$ and $\widehat{G_{2}}%
, $ then $\Bbb{G}_{1}$ and $\Bbb{G}_{2}$ are groupoid-isomorphic (See [10]
and [11]).\strut \strut \strut

Let $\mathcal{X}$ $=$ $(\mathcal{X},$ $\mathcal{Y},$ $s,$ $r)$ be a
groupoid. We say that this groupoid $\mathcal{X}$ \emph{acts on a set }$Y,$
if there exists a groupoid action $\pi $ of $\mathcal{X}$ such that: (i) $%
\pi (x)$ $:$ $Y$ $\rightarrow $ $Y$ is a well-defined function, for all $x$ $%
\in $ $\mathcal{X},$ and (ii) $\pi $ satisfies

\begin{center}
$\pi (x_{1}x_{2})$ $=$ $\pi (x_{1})$ $\circ $ $\pi (x_{2})$ on $Y,$
\end{center}

for all $x_{1},$ $x_{2}$ $\in $ $\mathcal{X}$, where ($\circ $) means the
usual composition of maps. We call the set $Y,$ a $\mathcal{X}$\emph{-set}.

Let $\mathcal{X}_{1}$ $\subset $ $\mathcal{X}_{2}$ be a subset, where $%
\mathcal{X}_{2}$ $=$ $(\mathcal{X}_{2},$ $\mathcal{Y}_{2},$ $s,$ $r)$ is a
groupoid. Assume that $\mathcal{X}_{1}$ $=$ $(\mathcal{X}_{1},$ $\mathcal{Y}%
_{1},$ $s,$ $r),$ itself, is a groupoid, where $\mathcal{Y}_{1}$ $=$ $%
\mathcal{X}_{1}$ $\cap $ $\mathcal{Y}_{2}.$ Then we say that the groupoid $%
\mathcal{X}_{1}$ is a \emph{subgroupoid} of $\mathcal{X}_{2}.$\strut \strut

Recall that we say a graph $G_{1}$ is a \emph{full-subgraph} of a countable
directed graph $G_{2},$ if\strut

\begin{center}
$E(G_{1})$ $\subseteq $ $E(G_{2})$
\end{center}

and

\begin{center}
$V(G_{1})$ $=$ $\{v$ $\in $ $V(G_{1})$ $:$ $e$ $=$ $v$ $e$ or $e$ $=$ $e$ $%
v, $ $\forall $ $e$ $\in $ $E(G_{1})\}.$\strut
\end{center}

Remark the difference between full-subgraphs and subgraphs: We say that $%
G_{1}^{\prime }$ is a \emph{subgraph} of $G_{2},$ if\strut

\begin{center}
$V(G_{1}^{\prime })$ $\subseteq $ $V(G_{2})$
\end{center}

and

\begin{center}
$E(G_{1}^{\prime })$ $=$ $\{e$ $\in $ $E(G_{2})$ $:$ $e$ $=$ $v_{1}$ $e$ $%
v_{2},$ for $v_{1},$ $v_{2}$ $\in $ $V(G_{1}^{\prime })\}.$\strut \strut
\end{center}

\strut Also, if a graph $V$ is a graph with $V(V)$ $\subseteq $ $V(G_{2}),$
and $E(V)$ $=$ $\varnothing ,$ then we call $V,$ a \emph{vertex subgraph of} 
$G_{2}.$

We will say that $G_{1}$ is a \emph{part} of $G_{2},$ if $G_{1}$ is either a
full-subgraph of $G_{2},$ or a subgraph of $G_{2},$ or a vertex subgraph of $%
G_{2}.$ It is easy to show that the graph groupoid $\Bbb{G}_{1}$ of $G_{1}$
is a subgroupoid of the graph groupoid $\Bbb{G}_{2}$ of $G_{2},$ whenever $%
G_{1}$ is a part of $G_{2}.$\strut

\strut

\textbf{Appendix B. Operator-Theoretic Properties}

\strut

Let $H$ be an arbitrary separable Hilbert space equipped with its inner
product $<,>.$ i.e., the inner product $<,>$ on $H$ is the sesquilinear form,

\begin{center}
$<,>$ $:$ $H$ $\times $ $H$ $\rightarrow $ $\Bbb{C},$
\end{center}

satisfying that:

(i)\ \ $\ <t_{1}\xi _{1}+t_{2}\xi _{2},$ $\eta >$ $=$ $t_{1}<\xi _{1},$ $%
\eta >$ $+$ $t_{2}$ $<\xi _{2},$ $\eta >,$

(ii)\  $<\xi ,$ $\eta >$ $=$ $\overline{<\eta ,\text{ }\xi >},$

(iii) $<\xi ,$ $\xi >$ $\;\geq $ $0,$ and equality holds, if and only if $%
\xi $ $=$ $0_{H},$

for all $\xi ,$ $\xi _{k},$ $\eta ,$ $\eta _{k}$ $\in $ $H,$ and $t_{k}$ $%
\in $ $\Bbb{C},$ where $\overline{t}$ mean the conjugates of $t,$ and $0_{H}$
means the zero vector in $H.$ As usual, let $B(H)$ be the operator algebra
consisting of all (bounded linear) operators on $H.$

For any operator $T$ $\in $ $B(H),$ there exists a unique operator $T^{*}$
satisfying

\begin{center}
$<T\xi ,$ $\eta >$ $=$ $<\xi ,$ $T^{*}\eta >,$
\end{center}

for all $\xi ,$ $\eta $ $\in $ $H.$ This operator $T^{*}$ is called the 
\emph{adjoint of} $T.$

\begin{definition}
Let $T$ $\in $ $B(H)$ be an operator.

(1) We say that an operator $T$ is \emph{self-adjoint}, if the adjoint $T^{*}
$ of\emph{\ }$T$ is identical to $T,$ i.e.,

\begin{center}
$T$ is self-adjoint $\overset{def}{\Longleftrightarrow }$ $T^{*}$ $=$ $T$ in 
$B(H).$
\end{center}

(2) An operator $T$ is said to be \emph{normal}, if the product $T^{*}T$ of $%
T^{*}$ and $T$ is identical to the product $TT^{*},$ on $H,$ i.e.,

\begin{center}
$T$ is normal $\overset{def}{\Longleftrightarrow }$ $T^{*}T$ $=$ $TT^{*}$ in 
$B(H).$
\end{center}

(3) We call $T$ a \emph{unitary}, if it is normal, and $T^{*}T$ $=$ $TT^{*}$
are identical to the identity operator $1_{H}$ on $H,$ i.e.,

\begin{center}
$T$ is unitary $\overset{def}{\Longleftrightarrow }$ $T^{*}T$ $=$ $1_{H}$ $=$
$TT^{*}$ in $B(H)$
\end{center}

\strut (4) An operator $T$ $\in $ $B(H)$ is called a \emph{projection}, if
it is self-adjoint and idempotent, in the sense that $T^{2}$ $=$ $T$ on $H.$
i.e.,

\begin{center}
$T$ is a projection $\overset{def}{\Longleftrightarrow }$ $T^{*}$ $=$ $T$ $=$
$T^{2}$ in $B(H).$
\end{center}

(5) We say an operator $T$ is \emph{positive} on $H,$ if

\begin{center}
$<T\xi ,$ $\xi >$ \ $\geq $ $0,$ for all $\xi $ $\in $ $H$ with $\left\| \xi
\right\| $ $=$ $1,$
\end{center}

where $\left\| \xi \right\| $ $\overset{def}{=}$ $\sqrt{<\xi ,\text{ }\xi >}$
is the Hilbert norm of $\xi $, for all $\xi $ $\in $ $H.$

(6) An operator $T$ is said to be \emph{hyponormal}, if the operator $T^{*}T$
$-$ $TT^{*}$ is positive on $H.$
\end{definition}

\strut Such properties of operators are well-known in operator theory. Also,
if an operator $T$ has one of the above properties, then it is a ``good''
operator in the theory. For instance, if $T$ is normal, then it satisfies
the spectral mapping theorem, and hence $f(T)$ is again normal, for all
continuous functions on $\Bbb{C},$ etc. By definition, we can check that:

(2.3.1) If $T$ is self-adjoint, then $T$ is normal.

(2.3.2) If $T$ is unitary, then $T$ is normal.

(2.3.3) Every projection is self-adjoint.

(2.3.4) Every normal operator is hyponormal.

(2.3.5) If $T$ is unitary, then $T$ is invertible, moreover, $T^{*}$ $=$ $%
T^{-1}.$

Clearly, the converses of the above facts does not hold true, in general.

We say that an operator $T$ is a \emph{partial isometry}, if the product $%
T^{*}T$ of the adjoint $T^{*}$ and $T$ is a projection. The following
characterization is also known: $T$ is a partial isometry, if and only if $%
TT^{*}T$ $=$ $T,$ if and only if $T^{*}$ is a partial isometry, if and only
if $T^{*}TT^{*}$ $=$ $T^{*}.$ In particular, the projections $T^{*}T$ and $%
TT^{*}$ are called the \emph{initial projection}, and the \emph{final
projection} of $T,$ respectively. i.e., the projection $T^{*}T$ (resp., the
projection $TT^{*}$) send the elements of $H$ into the elements of the
subspace $H_{init}^{T}$ (resp., $H_{fin}^{T}$) of $H.$ We call the subspaces 
$H_{init}^{T}$ and $H_{fin}^{T}$ of $H,$ induced by a partial isometry $T$,
the \emph{initial subspace} and the \emph{final subspace} of $T$ in $H,$
respectively.

A partial isometry $T$ satisfying that $T^{*}T$ $=$ $1_{H}$ is called an 
\emph{isometry}. Keep in mind that, even though $T^{*}T$ $=$ $1_{H},$ it is
possible that $TT^{*}$ $\neq $ $1_{H}.$ Clearly, if $TT^{*}$ $=$ $1_{H},$
for an isometry $T,$ then this isometry $T$ becomes a unitary. A partial
isometry $T,$ satisfying $TT^{*}$ $=$ $1_{H},$ is called a \emph{co-isometry}%
.

In many cases, the projections $T^{*}T$ and $TT^{*}$ of a partial isometry $%
T $ are distinct from each other, whenever $H_{init}^{T}$ and $H_{fin}^{T}$
are different in $H.$ This shows that partial isometries are not normal, in
general. For instance, let

\begin{center}
$T$ $=$ $\left( 
\begin{array}{ll}
0 & 0 \\ 
1 & 0
\end{array}
\right) ,$ with its adjoint $T^{*}$ $=$ $\left( 
\begin{array}{ll}
0 & 1 \\ 
0 & 0
\end{array}
\right) ,$
\end{center}

on $H$ $=$ $\Bbb{C}^{\oplus 2}.$ Then it is a partial isometry, since $%
TT^{*}T$ $=$ $T.$ And, it has its initial projection and final projection as
follows:

\begin{center}
$T^{*}T$ $=$ $\left( 
\begin{array}{ll}
1 & 0 \\ 
0 & 0
\end{array}
\right) $, and $TT^{*}$ $=$ $\left( 
\begin{array}{ll}
0 & 0 \\ 
0 & 1
\end{array}
\right) ,$
\end{center}

on $H.$ We can easily see that $T^{*}T$ $\neq $ $TT^{*},$ and hence $T$ is
not normal.

Let $T$ $\in $ $B(H)$ be an operator. Then $T$ has its \emph{spectrum }$%
spec(T),$ defined by a subset

\begin{center}
$spec(T)$ $\overset{def}{=}$ $\{t$ $\in $ $\Bbb{C}$ $:$ $T$ $-$ $t$ $1_{H}$
is not invertible on $H\},$
\end{center}

in $\Bbb{C}.$ It is well-known that every spectrum $spec(T)$ is nonempty and
compact in $\Bbb{C},$ whenever $T$ is a (bounded linear) operator on a
(complex) Hilbert space $H.$

This numerical data for $T$ is very valuable to analyze the operators. In
particular, every normal operator $T$ can be understood (or regarded) as a
complex-valued function, satisfying

\begin{center}
$\int_{spec(T)}$ $t$ $dE,$
\end{center}

where $E$ is the suitable (operator-valued) measure on $spec(T),$ called the%
\emph{\ spectral measure}. Thus, if $f$ is a $\Bbb{C}$-valued continuous
map, then

\begin{center}
$f(T)$ $=$ $\int_{spec(T)}$ $f(t)$ $dE(t).$
\end{center}

i.e., the \emph{spectral mapping theorem} holds for normal operators. Thus
all operators $f(T)$ are normal, too, whenever $T$ is normal.

Also, by Gelfand, the $C^{*}$-algebra $C^{*}(T),$ generated by $T,$ is $*$%
-isomorphic to the $C^{*}$-algebra $C(spec(T)),$ consisting of all
continuous $\Bbb{C}$-valued functions on $spec(T).$ (However, finding
spectra of operators is not easy at all.)

\strut \strut \strut 

\strut

\strut

\strut \textbf{References}\strut

\strut

\strut

{\small [1] \ \ A. Gibbons and L. Novak, Hybrid Graph Theory and Network
Analysis, ISBN: 0-521-46117-0, (1999) Cambridge Univ. Press.}

{\small [2]\strut \ \ \ \strut B. Solel, You can see the arrows in a Quiver
Operator Algebras, (2000), preprint.}

{\small [3] \ \ I. Cho, Hyponormality of Toeplitz Operators with
Trigonometric Polynomial Symbols, Master Degree Thesis, (1999) Sungkyunkwan
Univ.}

{\small [4] \ \ \strut I. Cho, Graph Groupoids and Partial Isometries, ISBN:
978-3-8383-1397-9, (2009) LAP Publisher.\strut }

{\small [5] \ \ \strut I. Cho, and P. E. T. Jorgensen, }$C^{*}${\small %
-Subalgebras Generated by Partial Isometries, JMP, DOI: 10.1063/1.3056588,
(2009).}

{\small [6] \ \ \strut I. Cho, and P. E. T. Jorgensen, }$C^{*}${\small %
-Subalgebras Generated by a Single Operator in }$B(H)${\small , ACTA Appl.
Math., 108, (2009) 625 - 664.}

{\small [7] \ \ \strut I. Cho, and P. E. T. Jorgensen, Measure Framings on
Graphs and Corresponding von Neumann Algebras, (2009) Preprint.}

{\small [8] \ \ \strut I. Raeburn, Graph Algebras, CBMS no 3, AMS (2005).}

{\small [9] \ \ \strut P. D. Mitchener, }$C^{*}${\small -Categories,
Groupoid Actions, Equivalent KK-Theory, and the Baum-Connes Conjecture,
arXiv:math.KT/0204291v1, (2005), Preprint.}

{\small [10] R. Gliman, V. Shpilrain and A. G. Myasnikov (editors),
Computational and Statistical Group Theory, Contemporary Math, 298, (2001)
AMS.}

{\small [11] F. Radulescu, Random Matrices, Amalgamated Free Products and
Subfactors of the $C^{*}$- Algebra of a Free Group, of Noninteger Index,
Invent. Math., 115, (1994) 347 - 389.}

{\small [12] P. R. Halmos. Hilbert Space Problem Book (2-nd Ed), ISBN:
0-387-90685-1, (1982) Springer-Verlag.}

{\small [13] T. Yosino, Introduction to Operator Theory, ISBN:
0-582-23743-2, (1993) Longman Sci. \& Tech.}

{\small [14] F. W. Stinespring, Positive Functions on }$C^{*}${\small %
-Algebras, Proc. Amer. Math. Soc., vol 6, (1955) 211 - 216.}

{\small [15] M. B. Stefan, Indecomposability of Free Group Factors over
Nonprime Subfactors and Abelian Subalgebras, Pacific J. Math., 219, no. 2,
(2005) 365 - 390.}

{\small [16] N. Tanaka, Conjugacy Classes of Zero Entropy Automorphisms on
Free Group Factors, Nihonkai M. J., 6, no. 2, (1995) 171 - 175. }

\end{document}